\documentclass{amsart}
\usepackage{amsfonts}
\usepackage{amsmath}
\usepackage{amssymb}
\usepackage{amsthm}
\usepackage{graphicx}
\usepackage{capt-of}
\usepackage{tikz}
\usetikzlibrary{patterns}
\usetikzlibrary{decorations.markings, patterns}
\usepackage{multirow}
\usepackage{array} 
\newcolumntype{C}{>{$}c<{$}}

\newtheorem{theorem}{Theorem}
\newtheorem{prop}{Proposition}
\newtheorem{lemma}{Lemma}
\newtheorem{rem}{Remark}
\newtheorem{exmp}{Example}

\begin{document}
\title{$Z$-oriented triangulations of surfaces}
\author{Adam Tyc}
\subjclass[2000]{}

\
\address{Adam Tyc: Institute of Mathematics, Polish Academy of Sciences, \'Sniadeckich 8, 00-656 Warszawa, Poland}
\email{atyc@impan.pl}

\maketitle
\begin{abstract}
The main objects of the paper are $z$-oriented triangulations of connected closed $2$-dimensional surfaces.
A $z$-orientation of a map is a minimal collection of zigzags which double covers the set of edges. 
We have two possibilities for an edge -- zigzags from the $z$-orientation pass through this edge in different directions (type I) or in the same direction (type II). 
Then there are two types of faces in a triangulation: the first type is when two edges of the face are of type I and one edge is of type II and the second type is when all edges of the face are of type II. 
We investigate $z$-oriented triangulations with all faces of the first type (in the general case, any $z$-oriented triangulation can be shredded to a $z$-oriented triangulation of such type).
A zigzag is homogeneous if it contains precisely two edges of type I after any edge of type II.
We give a topological characterization of the homogeneity of zigzags; 
in particular, we describe a one-to-one correspondence between $z$-oriented triangulations with homogeneous zigzags and closed $2$-cell embeddings of directed Eulerian graphs in surfaces. 
At the end, we give an application to one type of the $z$-monodromy.
\end{abstract}

\section{Introduction}
Petrie polygons are well-known objects described by Coxeter \cite{Coxeter} (see also \cite{McMSch}).
These are skew polygons in regular polyhedra such that  any two consecutive edges, but not three, are on the same face.
Analogs of Petrie polygons for graphs embedded in surfaces are called {\it zigzags} \cite{DDS-book,Lins1} or {\it closed left-right paths} \cite{GR-book,Shank}.
Zigzags have many applications, for example, they are successfully exploited to enumerate all combinatorial possibilities for fullerenes \cite{BD}.
The case when a map, i.e an embedding of a graph in a surface, has a single zigzag is very important \cite{DDS-book,GR-book}.
Following \cite{DDS-book} we call such maps $z$-{\it knotted}.
They have nice homological properties and are closely connected to Gauss code problem 
\cite{CrRos, GR-book,Lins2}.

The studying of zigzags in $3$-regular plane graphs, in particular fullerenes, is one of the main directions of \cite{DDS-book}.
A large class of $z$-knotted $3$-regular plane graphs is obtained by using computer.
The dual objects, i.e. spherical triangulations, have the same zigzag structure. 
Zigzags in triangulations of surfaces (not necessarily orientable) are investigated in \cite{PT1,PT2,PT3}.
By \cite{PT2}, every such triangulation admits a $z$-knotted shredding.

A $z$-{\it orientation} of a map is a minimal collection of zigzags which double covers the set of edges \cite{DDS-book}.
In the $z$-knotted case, this collection contains only one zigzag and is unique up to reversing.
For every $z$-orientation we have the following two types of edges:
an edge is of type I if the distinguished zigzags pass through this edge in different directions
and an edge is of type II if they pass through the edge in the same direction. 
It is not difficult to prove that for every face in a triangulation with fixed $z$-orientation one of the following possibilities is realized:
the face contains precisely two edges of type I and the third edge is of type II or all edges are of type II.
We observe that every $z$-oriented triangulation can be shredded to a triangulation where all faces are of the first type. 
In this paper, we restrict ourself to $z$-oriented triangulations with all faces of type I. 

Let $\Gamma$ be such a triangulation of a surface $M$. 
Then the number of edges of type I is the twofold number of edges of type II
and we say that a zigzag is {\it homogeneous} if it contains precisely two edges of type I after each edge of type II. 
Denote by $\Gamma_{II}$ the subgraph of $\Gamma$ formed by all edges of type II. 
Our first result (Theorem \ref{newth1}) states that the following three conditions are equivalent:
\begin{enumerate}
\item[{\rm (1)}] all zigzags of $\Gamma$ are homogeneous,
\item[{\rm (2)}] $\Gamma_{II}$ is a closed 2-cell embedding of a simple Eulerian digraph such that every face is a directed cycle,
\item[{\rm (3)}] each connected component of $M\setminus\Gamma_{II}$ is homeomorphic to an open 2-dimensional disk.
\end{enumerate}
Note that directed Eulerian spherical embeddings are known also as  {\it plane alternating dimaps};
they are investigated, for example, in \cite{BHS, Farr, McCourt}.
Directed Eulerian embeddings in arbitrary surfaces are considered in \cite{BCMMcK, CGH}.

We will use the following structural property of $\Gamma$ (without assumption that the zigzags are homogeneous): 
the connected components of $M\setminus\Gamma_{II}$ are open disks, cylinders or M\"obius strips (the third type of components can be realized only for the non-orientable case) and all these possibilities are realized. 
We show that the existence of cylinders or M\"obius strips contradicts to the homogeneity of zigzags.

By \cite{PT2}, there are precisely $7$ types of $z$-monodromies (M1)--(M7). 
For each of the types (M3)--(M5) and (M7) there is a triangulation such that each face has the $z$-monodromy of this type.
The types (M1) and (M2) are exceptional: all faces with $z$-monodromies of each of these types form a forest \cite{PT3}. 
The case (M6) cannot be investigated by the methods of \cite{PT3} and the authors left it as an open problem.
It is easy to see that each face with the $z$-monodromy (M6) is of the first type for every $z$-orientation. 
Using this fact, we construct a series of toric triangulations where all faces have $z$-monodromies of type (M6).

\section{Zigzags and $z$-orientations of triangulations of surfaces}
Let $M$ be a connected closed $2$-dimensional surface (not necessarily orientable).
A {\it triangulation} of $M$ is a $2$-cell embedding of a connected simple finite graph in $M$ such that all faces are triangles \cite[Section 3.1]{MT-book}.
Then the following assertions are fulfilled:
(1) every edge is contained in precisely  two distinct faces,
(2) the intersection of two distinct faces is an edge or a vertex or empty.

Let $\Gamma$ be a triangulation of $M$.
A {\it zigzag}  in $\Gamma$ is a sequence of edges $\{e_{i}\}_{i\in {\mathbb N}}$ satisfying the following conditions for every natural $i$: 
\begin{enumerate}
\item[$\bullet$] $e_{i}$ and $e_{i+1}$ are distinct edges of a certain face 
(then they have a common vertex, since every face is a triangle),
\item[$\bullet$] the faces containing $e_{i},e_{i+1}$ and $e_{i+1},e_{i+2}$ are distinct
and the edges $e_{i}$ and $e_{i+2}$ are non-intersecting.
\end{enumerate} 
Since $\Gamma$ is finite, 
there is a natural number $n>0$ such that $e_{i+n}=e_{i}$ for every natural $i$. 
In what follows, every zigzag will be presented as a cyclic sequence $e_{1},\dots,e_{n}$,
where $n$ is the smallest number satisfying the above condition.

Every zigzag is completely determined by any pair of consecutive edges  belonging to this zigzag
and for any distinct edges $e$ and $e'$ on a face there is a unique zigzag containing the sequence $e,e'$.
If $Z=\{e_{1},\dots,e_{n}\}$ is a zigzag, then the reversed sequence $Z^{-1}=\{e_{n},\dots,e_{1}\}$ also is a zigzag.
A zigzag cannot contain a sequence $e,e',\dots,e',e$ which implies that
$Z\ne Z^{-1}$ for any zigzag $Z$, i.e. a zigzag cannot be self-reversed (see, for example, \cite{PT2}).
We say that $\Gamma$ is $z$-{\it knotted} if it contains precisely two zigzags $Z$ and $Z^{-1}$, in other words,
there is a single zigzag up to reversing.

\begin{exmp}{\rm
Zigzags in the Platonic solids (three of them are triangulations of the sphere) are skew polygons without self-intersections and are called
 {\it  Petrie polygons}.
}\end{exmp}
\begin{exmp}{\rm
Let $BP_n$ be  the $n$-gonal bipyramid, where $1,\dots,n$ denote the consecutive vertices of the base 
and the remaining two vertices are denoted by $a,b$ (see Fig.1 for $n=3$).  
\begin{center}
\begin{tikzpicture}[scale=0.3]


\coordinate (A) at (0.5,4);
\coordinate (B) at (0.5,-5);
\coordinate (1) at (0,-1);
\coordinate (2) at (-2,0);
\coordinate (3) at (3,0);

\draw[fill=black] (A) circle (3.5pt);
\draw[fill=black] (B) circle (3.5pt);

\draw[fill=black] (1) circle (3.5pt);
\draw[fill=black] (2) circle (3.5pt);
\draw[fill=black] (3) circle (3.5pt);

\draw[thick] (3)--(1)--(2);
\draw[thick, dashed] (2)--(3);

\draw[thick] (A)--(1);
\draw[thick] (A)--(2);
\draw[thick] (A)--(3);

\draw[thick] (B)--(1);
\draw[thick] (B)--(2);
\draw[thick] (B)--(3);

\node at (3.6,0) {$3$};
\node at (-0.45,-1.525) {$2$};
\node at (-2.5,0) {$1$};

\node at (0.5,4.6) {$a$};
\node at (0.5,-5.8) {$b$};

\node[color=white] at (4,0) {$.$};

\end{tikzpicture}
\captionof{figure}{ }
\end{center}
(a). In the case when $n=2k+1$, the bipyramid $BP_n$ is $z$-knotted.
If $k$ is odd, then the unique (up to reversing) zigzag is 
$$a1,12,2b,b3,\dots,a(n-2),(n-2)(n-1),(n-1)b,bn,n1,$$
$$1a,a2,23,3b,\dots, a(n-1),(n-1)n,nb,$$
$$b1,12,2a,a3,\dots,b(n-2),(n-2)(n-1),(n-1)a,an,n1,$$
$$1b,b2,23,3a,\dots,b(n-1),(n-1)n, na.$$
If $k$ is even, then this zigzag is
$$a1,12,2b,b3,\dots,b(n-2),(n-2)(n-1), (n-1)a,an,n1,$$
$$1b,b2,23,3a,\dots,a(n-1),(n-1)n,nb,$$
$$b1,12,2a,a3,\dots,a(n-2),(n-2)(n-1),(n-1)b,bn,n1,$$
$$1a,a2,23,3b,\dots,b(n-1),(n-1)n, na.$$
(b). If $n=2k$ and $k$ is odd, 
then the bipyramid contains precisely two zigzags (up to reversing):
$$a1,12,2b,b3,34,\dots,a(n-1),(n-1)n, nb,$$
$$b1,12,2a,a3,34,\dots,b(n-1),(n-1)n, na$$
and 
$$a2,23,3b,b4,45,\dots,an,n1,1b,$$
$$b2,23,3a,a4,45,\dots,bn,n1,1a.$$
(c). In the case when $n=2k$ and $k$ is even, 
there are precisely four zigzags (up to reversing):
$$a1,12, 2b,\dots,b(n-1),(n-1)n, na;$$ 
$$b1,12, 2a,\dots,a(n-1),(n-1)n, nb;$$
$$a2,23,3b,\dots,bn,n1,1a;$$
$$b2,23,3a,\dots,an,n1,1b.$$
}\end{exmp}
See \cite{PT1,PT2} for more examples of $z$-knotted triangulations. 
Examples of $z$-knotted fullerenes can be found in \cite{DDS-book}.

Suppose that $\Gamma$ contains precisely distinct $k$ zigzags up to reversing. 
A $z$-{\it orien\-tation} of $\Gamma$ is a collection $\tau$ consisting of $k$ distinct  zigzags  such that 
for each zigzag $Z$ we have $Z\in \tau$ or $Z^{-1}\in \tau$. 
There are precisely $2^k$ distinct $z$-orientations of $\Gamma$.
For every $z$-orientation $\tau=\{Z_{1},\dots,Z_{k}\}$ the $z$-orientation $\tau^{-1}=\{Z^{-1}_{1},\dots, Z^{-1}_{k}\}$ will be called {\it reversed} to $\tau$. 
The triangulation $\Gamma$ with a $z$-orientation $\tau$ will be denoted by $(\Gamma, \tau)$ and called a {\it $z$-oriented triangulation}.

Let $\tau$ be a $z$-orientation of $\Gamma$. 
For every edge $e$ of $\Gamma$ one of the following possibilities is realized:
\begin{enumerate}
\item[$\bullet$] there is a zigzag $Z\in \tau$ such that $e$ occurs in this zigzag twice and other zigzags from $\tau$
do not contain $e$,
\item[$\bullet$] there are two distinct zigzags $Z,Z'\in\tau$ such that $e$ occurs in each of these zigzags only ones 
and other zigzags from $\tau$ do not contain $e$. 
\end{enumerate}
In the first case, we say that $e$ is an {\it edge of type} I if $Z$ passes through $e$ twice in different directions; 
otherwise, $e$ is said to be an {\it edge of type} II.
Similarly, in the second case:
$e$ is an {\it edge of type} I if $Z$ and $Z'$ pass through $e$ in different directions
or $e$ is an {\it edge of type} II if $Z$ and $Z'$ pass through $e$ in the same direction.
In what follows, edges of type II will be considered together with the direction defined by $\tau$.
A vertex of $\Gamma$ is called {\it of type} I if it belongs only to edges of type I;
otherwise, we say that this is a {\it vertex of type} II.

The following statements hold for any $z$-orientation $\tau$ of $\Gamma$.

\begin{lemma} \label{lemma1}
For each vertex of type {\rm II} the number of edges of type {\rm II} which enter this vertex
is equal to the number of edges of type {\rm II} which leave it.
\end{lemma}

\begin{proof}
The number of times that the zigzags from $\tau$ enter a vertex  is equal to the number of times that these zigzags leave this vertex.
\end{proof}

\begin{prop}\label{prop-or}
For every face of $(\Gamma, \tau)$ one of the following possibilities is realized:
\begin{enumerate}
\item[{\rm (I)}] the face contains two edges of type {\rm I} and the third edge is of type {\rm II}, see {\rm Fig.2(a)};
\item[{\rm (II)}] all edges of the face are of type {\rm II} and form a directed cycle, see {\rm Fig.2(b)}.
\end{enumerate}
\end{prop}
A face in a triangulation is said to be {\it of type} I or {\it of type} II if the corresponding possibility is realized.
\begin{center}
\begin{tikzpicture}[scale=0.6]

\draw[fill=black] (0,2) circle (3pt);
\draw[fill=black] (-1.7320508076,-1) circle (3pt);
\draw[fill=black] (1.7320508076,-1) circle (3pt);

\draw [thick, decoration={markings,
mark=at position 0.62 with {\arrow[scale=1.5,>=stealth]{>>}}},
postaction={decorate}] (5.1961524228,2) -- (3.4641016152,-1);

\draw [thick, decoration={markings,
mark=at position 0.62 with {\arrow[scale=1.5,>=stealth]{>>}}},
postaction={decorate}] (3.4641016152,-1) -- (6.9282032304,-1);

\draw [thick, decoration={markings,
mark=at position 0.62 with {\arrow[scale=1.5,>=stealth]{<<}}},
postaction={decorate}] (5.1961524228,2) -- (6.9282032304,-1);

\node at (0,-1.65) {(a)};

\draw[fill=black] (5.1961524228,2) circle (3pt);
\draw[fill=black] (3.4641016152,-1) circle (3pt);
\draw[fill=black] (6.9282032304,-1) circle (3pt);

\draw [thick, decoration={markings,
mark=at position 0.62 with {\arrow[scale=1.5,>=stealth]{><}}},
postaction={decorate}] (0,2) -- (-1.7320508076,-1);

\draw [thick, decoration={markings,
mark=at position 0.62 with {\arrow[scale=1.5,>=stealth]{>>}}},
postaction={decorate}] (-1.7320508076,-1) -- (1.7320508076,-1);

\draw [thick, decoration={markings,
mark=at position 0.62 with {\arrow[scale=1.5,>=stealth]{><}}},
postaction={decorate}] (0,2) -- (1.7320508076,-1);

\node at (5.1961524228,-1.65) {(b)};
\end{tikzpicture}
\captionof{figure}{ }
\end{center}

\begin{proof}[Proof of Proposition \ref{prop-or}]
Consider a face whose edges are denoted by $e_{1},e_{2},e_{3}$.
Without loss of generality we can assume that the zigzag containing the sequence $e_{1},e_{2}$ belongs to $\tau$.
Let $Z$ and $Z'$ be the zigzags containing the sequences $e_{2},e_{3}$ and $e_{3},e_{1}$, respectively. 
Then $Z\in \tau$ or $Z^{-1}\in \tau$ and $Z'\in \tau$ or $Z'^{-1}\in \tau$.
An easy verification shows that for each of these four cases we obtain (I) or (II).
\end{proof}

\begin{exmp}{\rm
If $n$ is odd, then the bipyramid $BP_n$ has the unique $z$-orientation (up to reversing), see Example 2(a).
The edges $ai$ and $bi$, $i\in \{1,\dots,n\}$ are of type I and the edges on the base of the bipyramid  are of type II.
The vertices $a,b$ are of type I and the vertices on the base are of type II. All faces are of type I.
The same happens for the case when $n=2k$ and $k$ is odd if the $z$-orientation is defined by 
the two zigzags presented in  Example 2(b); 
however, all faces are of type II  if we replace one of these zigzags by the reversed.
}\end{exmp}

\begin{exmp}{\rm
Suppose that $n=2k$ and $k$ is even.
Let $Z_{1},Z_{2},Z_{3},Z_{4}$ be the zigzags from Example 2(c).
For the $z$-orientation defined by these zigzags  all faces are of type I.
If the $z$-orientation is defined by $Z_{1},Z_{2}$ and $Z^{-1}_{3}, Z^{-1}_{4}$,
then all faces are of type II. 
In the case when the $z$-orientation is defined by $Z_{1},Z_{2},Z_{3}$ and $Z^{-1}_{4}$,
there exist faces of the both types.
}\end{exmp}

\begin{rem}{\rm
If we replace a $z$-orientation by the reversed $z$-orientation,
then the type of every edge does not change (but all edges of type II reverse the directions), consequently, 
the types of vertices and faces also do not change.
For $z$-knotted triangulations we can say about the types of edges, vertices and faces without attaching to a $z$-orientation \cite{PT1}.
}\end{rem}

A triangulation $\Gamma'$ of $M$ is a {\it shredding} of the triangulation $\Gamma$ if it is obtained from $\Gamma$ by triangulating some faces of $\Gamma$ such that all new vertices are contained in the interiors of these faces.

\begin{prop}\label{prop-sh}
Any $z$-oriented triangulation admits a $z$-oriented shredding with all faces of type I.
\end{prop}
\begin{proof}
Let $F$ be a face of type II in a $z$-oriented triangulation $(\Gamma,\tau)$ and let $e_1, e_2, e_3$ be edges of $F$. 
Suppose that the edges of $F$ are oriented as in Fig. 3 and denote by $\sigma$ the permutation $(1,2,3)$.
\begin{center}
\begin{tikzpicture}[scale=1.4]

\begin{scope}[xshift=-3cm]
\draw[fill=black] (0,0) circle (1.5pt);
\draw[fill=black] (-60:2cm) circle (1.5pt);
\draw[fill=black] (-120:2cm) circle (1.5pt);
\draw[thick, line width=2pt] (-60:2cm) -- (-120:2cm);

\draw[thick, line width=2pt] (-60:2cm) -- (0,0) -- (-120:2cm);

\draw [thick, decoration={markings,
mark=at position 0.51 with {\arrow[scale=2,>=stealth]{>}},
mark=at position 0.59 with {\arrow[scale=2,>=stealth]{>}}},
postaction={decorate}] (-60:2cm) -- (0,0);

\draw [thick, decoration={markings,
mark=at position 0.51 with {\arrow[scale=2,>=stealth]{>}},
mark=at position 0.59 with {\arrow[scale=2,>=stealth]{>}}},
postaction={decorate}] (0,0) -- (-120:2cm);

\draw [thick, decoration={markings,
mark=at position 0.51 with {\arrow[scale=2,>=stealth]{>}},
mark=at position 0.59 with {\arrow[scale=2,>=stealth]{>}}},
postaction={decorate}] (-120:2cm) -- (-60:2cm);

\node at (-135:1cm) {$e_3$};
\node at (-47:1cm) {$e_2$};
\node at (0,-1.95cm) {$e_1$};

\end{scope}

\draw [->]  (-2cm,-0.75) -- (-1cm,-0.75);

\draw[fill=black] (0,0) circle (1.5pt);
\draw[fill=black] (-60:2cm) circle (1.5pt);
\draw[fill=black] (-120:2cm) circle (1.5pt);
\draw[thick, line width=2pt] (-60:2cm) -- (-120:2cm);

\draw[thick, line width=2pt] (-60:2cm) -- (0,0) -- (-120:2cm);

\draw [thick, decoration={markings,
mark=at position 0.51 with {\arrow[scale=2,>=stealth]{>}},
mark=at position 0.59 with {\arrow[scale=2,>=stealth]{>}}},
postaction={decorate}] (-60:2cm) -- (0,0);

\draw [thick, decoration={markings,
mark=at position 0.51 with {\arrow[scale=2,>=stealth]{>}},
mark=at position 0.59 with {\arrow[scale=2,>=stealth]{>}}},
postaction={decorate}] (0,0) -- (-120:2cm);

\draw [thick, decoration={markings,
mark=at position 0.51 with {\arrow[scale=2,>=stealth]{>}},
mark=at position 0.59 with {\arrow[scale=2,>=stealth]{>}}},
postaction={decorate}] (-120:2cm) -- (-60:2cm);

\draw[thick, line width=1pt, dashed] (-90:1.1547cm) -- (0,0);

\draw [thick, line width=1pt, dashed, decoration={markings,
mark=at position 0.45 with {\arrow[scale=2]{<}},
mark=at position 0.6 with {\arrow[scale=2]{>}}},
postaction={decorate}] (-90:1.1547cm) -- (-120:2cm);

\draw [thick, line width=1pt, dashed, decoration={markings,
mark=at position 0.45 with {\arrow[scale=2]{<}},
mark=at position 0.6 with {\arrow[scale=2]{>}}},
postaction={decorate}] (-90:1.1547cm) -- (-60:2cm);

\draw [thick, line width=1pt, dashed, decoration={markings,
mark=at position 0.45 with {\arrow[scale=2]{<}},
mark=at position 0.6 with {\arrow[scale=2]{>}}},
postaction={decorate}] (-90:1.1547cm) -- (0,0);

\draw[fill=white] (-90:1.1547cm) circle (1.5pt);

\node at (-135:1cm) {$e_3$};
\node at (-47:1cm) {$e_2$};
\node at (0,-1.95cm) {$e_1$};

\node at (-75:1.17cm) {$e'_3$};
\node at (-105:1.17cm) {$e'_2$};
\node at (-81:0.9cm) {$e'_1$};

\end{tikzpicture}
\captionof{figure}{ }
\end{center}
Zigzags from $\tau$ passes through $F$ precisely three times, so the face $F$ separates them into $3$ segments of type
$$e_{\sigma^{-1}(i)},e_{i},X_{ij},e_j,e_{\sigma(j)},$$ 
where $i,j\in\{1,2,3\}$ and the sequence $X_{ij}$ is a maximal part of a zigzag formed by edges occuring between $e_i$ and $e_j$. 
Let $\mathcal{X}$ be the set of all such sequences $X_{ij}$ for $F$ and the $z$-orientation $\tau$. 
Note that every $X_{ij}\in\mathcal{X}$ is completely determined by the beginning edge $e_i$ and the final edge $e_j$.
Now, we triangulate the face $F$ by adding a vertex in the interior of $F$ and three edges  connecting this vertex with the vertices of $F$.
We denote this new triangulation by $\Gamma'$ and write $e'_i$ for the new edge if it does not has a common vertex with $e_i$ (see Fig. 3). 
Observe that for any $i\in\{1,2,3\}$ there exists a zigzag in $\Gamma'$ containing a subsequence of the form
$$e_i, e'_{\sigma^{-1}(i)},e'_i,e_{\sigma^{-1}(i)},X_{\sigma^{-1}(i)j}$$
for certain $j\in\{1,2,3\}$ and $X_{\sigma^{-1}(i)j}\in\mathcal{X}$. 
The edge $e_j$ which occurs in the zigzag directly after this subsequence is the same as the edge after $X_{\sigma^{-1}(i)j}$ in $(\Gamma,\tau)$, since $X_{\sigma^{-1}(i)j}$ does not contain edges of $F$.
Therefore, zigzags of $\Gamma'$ related to the three faces not contained in $\Gamma$ pass through the edges coming from $\Gamma$ in the same way as zigzags from $\tau$.
This implies the existence of a $z$-orientation of $\Gamma'$ such that all edges from $\Gamma$ do not change their types and the three new faces of $\Gamma'$ contained in $F$ are of type I. 
Recursively, we eliminate all faces of type II from $(\Gamma,\tau)$ and come to a $z$-oriented shredding of $\Gamma$ with all faces of type I and such that the type of any edge from $(\Gamma,\tau)$ is preserved.
\end{proof}

\section{Homogeneous zigzags in triangulations with faces of type I}
In this section, we will always suppose that $\Gamma$ is a triangulation with fixed $z$-orientation $\tau$ such that all faces in $\Gamma$ are of type I,
i.e. each face contains precisely two edges of type I and the third edge is of type II. 
If $m$ is the number of faces, then there are precisely $m$ edges of type I and $m/2$ edges of type II.
In other words, the number of edges of type I is the twofold number of edges of type II.
We say that a zigzag of $\Gamma$ is {\it homogeneous} if it is a cyclic sequence $\{e_{i},e'_{i},e''_{i}\}^{n}_{i=1}$,
where each $e_{i}$ is an edge of type II and all $e'_{i},e''_{i}$ are edges of type I.
If a zigzag is homogeneous, then the reversed zigzag also is homogeneous.  
Denote by $\Gamma_{II}$ the subgraph of $\Gamma$ formed by all vertices and edges of type II.

\begin{exmp}{\rm
The zigzags of $\Gamma=BP_n$ are homogeneous if $n$ is odd (the $z$-knotted case) or
$n$ is even and the $z$-orientation is defined by the two zigzags from Example 2(b) or by the four zigzags from Example 2(c). 
Only $a$ and $b$ are vertices of type I and $\Gamma_{II}$ is the directed cycle formed by the edges of the base of the bipyramid. 
Conversely, if all zigzags of $\Gamma$ are homogeneous and there are precisely two vertices of type I,
then $\Gamma$ is a bipyramid (easy verification). 
}\end{exmp}

\begin{exmp}\label{ex_6}{\rm
Let $\Gamma'$ be a triangulation of $M$ with a $z$-orientation such that all faces are of type II (see \cite[Example 4]{PT3} for a $z$-knotted triangulation of ${\mathbb S}^2$ whose faces are of type II). 
As in the proof of Proposition \ref{prop-sh}, we consider the shredding $\Gamma''$ of $\Gamma'$ which is obtained by adding a vertex in the interior of each face and three edges connecting this vertex with the vertices of the face. 
This triangulation $\Gamma''$ admits a $z$-orientation such that all faces are of type I. 
Every zigzag $e_{1},e_{2},e_{3},\dots$ in $\Gamma'$ is extended to a zigzag
$$e_{1},e'_{1},e''_{1},e_{2},e'_{2},e''_{2},e_{3},\dots$$
in $\Gamma''$ which passes through edges of $\Gamma'$ in the opposite directions. 
All $e_i$ are of type II and all $e_i'$ are of type I. 
So, all zigzags in $\Gamma''$ are homogeneous.
}\end{exmp}

\begin{theorem}\label{newth1}
The following three conditions are equivalent:
\begin{enumerate}
\item[{\rm (1)}] All zigzags of $\Gamma$ are homogeneous.
\item[{\rm (2)}] $\Gamma_{II}$ is a closed 2-cell embedding of a simple Eulerian digraph such that every face is a directed cycle.
\item[{\rm (3)}] Each connected component of $M\setminus\Gamma_{II}$ is homeomorphic to an open 2-dimensional disk.
\end{enumerate}
\end{theorem}
The implication (2) $\Rightarrow$ (3) is obvious. 
The implications (1) $\Rightarrow$ (2) and (3) $\Rightarrow$ (1) will be proved in Section 4 and Section 5, respectively. 

\section{Proof of the implication (1) $\Rightarrow$ (2) in Theorem 1}

Now, we generalize the construction described in Proposition \ref{prop-sh} and Example \ref{ex_6}.
Let $\Gamma'$ be a closed $2$-cell embedding of a connected finite simple graph in the surface $M$.
Then all faces of $\Gamma'$ are homeomorphic to a closed $2$-dimensional disk. 
For each face $F$ we take a point $v_{F}$ belonging to the interior of $F$.
We add all $v_{F}$ to the vertex set of $\Gamma'$ and connect each $v_{F}$ with every vertex of $F$ by an edge.
We obtain a triangulation of $M$ which will be denoted by ${\rm T}(\Gamma')$.

The assumption that our $2$-cell embedding is closed cannot be omitted.
Indeed, if a certain face of $\Gamma'$ is not homeomorphic to a closed $2$-dimensional disk, 
then there is a pair of vertices connected by a double edge and ${\rm T}(\Gamma')$ is not a triangulation in our sense.

\begin{prop}\label{theorem1}
If all zigzags of $\Gamma$ are homogeneous, 
then $\Gamma_{II}$ is a closed $2$-cell embedding of a simple Eulerian digraph 
such that every face is a directed cycle and $\Gamma={\rm T}(\Gamma_{II})$.
Conversely, if $\Gamma'$ is a closed $2$-cell embedding of a simple Eulerian digraph and every face is a directed cycle,
then the triangulation ${\rm T}(\Gamma')$ admits a unique $z$-orien\-tation 
such that all zigzags of ${\rm T}(\Gamma')$ are homogeneous
and $\Gamma'$ is a subgraph of ${\rm T}(\Gamma')$ formed by all vertices and edges of type II. 
\end{prop}

\begin{proof}
(I). 
Let $v$ be a vertex of $\Gamma$. 
Consider all faces containing $v$ and take the edge on each of these faces which does not contain $v$.
All such edges form a cycle which will be denoted by $C(v)$.

Suppose that all zigzags of $\Gamma$ are homogeneous and
consider any edge $e_{1}$ of type II. 
Let $v_{1}$ and $v_{2}$ be the vertices of this edge such that $e_{1}$ is directed from $v_{1}$ to $v_{2}$.
We choose one of the  two faces containing $e_1$ and take in this face the vertex $v$ which does not belong to $e_1$.
Let $e'_{1}$ and $e''_{1}$ be the edges which contain $v$ and occur in a certain zigzag $Z\in \tau$ immediately after $e_{1}$, see Fig.4.
Denote by $e_2$ the third edge of the face containing $e'_{1}$ and $e''_{1}$.
This edge contains $v_{2}$ and another one vertex, say $v_{3}$.
Since $Z$ is homogeneous, the edges $e'_{1}$ and $e''_{1}$ are of type I, and consequently, $e_2$ is of type II.
The zigzag  which goes through $e'_1$ from $v$ to $v_{2}$ belongs to $\tau$
(this follows easily from the fact that $Z$ goes through $e'_{1}$ in the opposite direction and $e'_1$ is an edge of type I).
The latter guarantees that
the edge $e_{2}$ is directed from $v_{2}$ to $v_{3}$.
By our assumption, the edge $e_{3}$ which occurs in $Z$ immediately after  $e'_{1}$ and $e''_{1}$ is of type II.
This edge is directed from $v_{3}$ to a certain vertex $v_{4}$.
So, $e_{1},e_{2},e_{3}$ are consecutive edges of the cycle $C(v)$ and each $e_i$ is directed from $v_i$ to $v_{i+1}$.
Consider the zigzag from $\tau$ which contains the sequence $e_2, e''_1$. 
The next edge in this zigzag connects $v$ and $v_{4}$ (the zigzag goes from $v$ to $v_{4}$).
Let $e_{4}$ be the edge which occurs in the zigzag after it. 
Then $e_{4}$ is an edge of type II (by our assumption), it belongs to $C(v)$ and leaves $v_{4}$.
Recursively, we establish that $C(v)$ is a directed cycle formed by edges of type II and every edge containing $v$ is of type I,
i.e. $v$ is a vertex of type I. 
Now, we consider the other face containing $e_1$ and take the vertex $v'$ of this face which does not belong to $e_{1}$. 
Using the same arguments, we establish that $v'$ is a vertex of type I and $C(v')$ is a directed  cycle formed by edges of type II. 
\begin{center}
\begin{tikzpicture}[scale=0.85, xshift=0.25cm]
\draw[fill=black] (0,0) circle (1.5pt);
\draw[fill=black] (0:2cm) circle (1.5pt);
\draw[fill=black] (-45:2cm) circle (1.5pt);
\draw[fill=black] (-90:2cm) circle (1.5pt);
\draw[fill=black] (-135:2cm) circle (1.5pt);
\draw[fill=black] (-180:2cm) circle (1.5pt);

\draw[thick, line width=2pt] (0:2cm) -- (-45:2cm) -- (-90:2cm) -- (-135:2cm) -- (-180:2cm);

\draw [thick, decoration={markings,
mark=at position 0.49 with {\arrow[scale=2,>=stealth]{>}},
mark=at position 0.61 with {\arrow[scale=2,>=stealth]{>}}},
postaction={decorate}] (-180:2cm) -- (-135:2cm);

\draw [thick, decoration={markings,
mark=at position 0.49 with {\arrow[scale=2,>=stealth]{>}},
mark=at position 0.61 with {\arrow[scale=2,>=stealth]{>}}},
postaction={decorate}] (-135:2cm) -- (-90:2cm);

\draw [thick, decoration={markings,
mark=at position 0.49 with {\arrow[scale=2,>=stealth]{>}},
mark=at position 0.61 with {\arrow[scale=2,>=stealth]{>}}},
postaction={decorate}] (-90:2cm) -- (-45:2cm);

\draw [thick, decoration={markings,
mark=at position 0.49 with {\arrow[scale=2,>=stealth]{>}},
mark=at position 0.61 with {\arrow[scale=2,>=stealth]{>}}},
postaction={decorate}] (-45:2cm) -- (0:2cm);

\draw[thick, line width=1pt, dashed] (-180:2cm) -- (0,0) -- (0:2cm);
\draw[thick, line width=1pt, dashed] (-45:2cm) -- (0,0);
\draw[thick, line width=1pt, dashed] (-90:2cm) -- (0,0);
\draw[thick, line width=1pt, dashed] (-135:2cm) -- (0,0);

\node at (0.15cm,0.15cm) {$v$};

\node at (-45:2.3cm) {$v_4$};
\node at (-90:2.3cm) {$v_3$};
\node at (-135:2.3cm) {$v_2$};
\node at (-180:2.3cm) {$v_1$};

\node at (-157.5:2.17cm) {$e_1$};
\node at (-112.5:2.17cm) {$e_2$};
\node at (-67.5:2.17cm) {$e_3$};
\node at (-22.5:2.17cm) {$e_4$};

\node at (-148:1.1cm) {$e'_1$};
\node at (-102:1.1cm) {$e''_1$};

\node[color=white] at (0:2.45cm) {s};

\end{tikzpicture}
\captionof{figure}{ }
\end{center}
For every vertex $v$ of type I we can take a face containing $v$ and the edge of this face which does not contain $v$.
This edge is of type II (since the remaining two edges of the face are of type I).
The above arguments show that the following assertions are fulfilled:
\begin{enumerate}
\item[(1)] vertices of type {\rm I} exist and for every such vertex $v$ the cycle $C(v)$ is a directed cycle formed by edges of type {\rm II};
\item[(2)] for every edge of type {\rm II} there are precisely two vertices $v$ and $v'$ of type {\rm I} such that 
this edge is contained in the cycles $C(v)$ and $C(v')$.
\end{enumerate}
Similarly, for every edge $e$ of type I we take a face containing $e$;
this face contains an edge of type II which implies that $e$ connects a vertices of different types.  

Consider $\Gamma_{II}$. 
Observe that any two vertices of type II in $\Gamma$ can be connected by a path formed by edges of type II
which means that $\Gamma_{II}$ is connected.
It is easy to see that $\Gamma_{II}$ is a $2$-cell embedding of a simple digraph 
such that every face is the directed cycle $C(v)$ for a certain vertex $v$ of type I;
in particular, this $2$-cell embedding is closed. 
Lemma 1 implies that $\Gamma_{II}$ is an Eulerian digraph. 
The equality $\Gamma={\rm T}(\Gamma_{II})$ is obvious.

The following remark will be used to prove the second part of the theorem.
The conditions (1) and (2) guarantee that every zigzag of $\Gamma$ containing an edge of type II is homogeneous.
Recall that the number of edges of type I is the twofold number of edges of type II.
This implies that there is no zigzag containing edges of type I only
(since every edge occurs twice in a unique zigzag from $\tau$ or it occurs ones in precisely two distinct zigzags from $\tau$).
Therefore, every zigzag of $\Gamma$ is homogeneous if (1) and (2)  hold.
 
 (II). Suppose that $\Gamma'$ is a closed $2$-cell embedding of a simple Eulerian digraph such that every face is a directed cycle.
 
Let $e_1,\dots,e_n$ be the directed cycle formed by all edges of a certain face of $\Gamma'$.
For every $i\in \{1,\dots,n\}$ we define $j(i)=i+2({\rm mod}\, n)$ and
denote by $e'_i$ and $e''_{i}$ the edges containing the vertex $v_{F}$ and intersecting $e_{i}$ and $e_{j(i)}$, respectively.
Consider the zigzag of ${\rm T}(\Gamma')$ which contains the sequence $e_{i},e'_{i},e''_{i}, e_{j(i)}$. 
It passes through $e_{i}$ and $e_{j(i)}$ according to the directions of these edges;
and the same holds for every edge of $\Gamma'$ which occurs in this zigzag.
Such a zigzag exists for any pair formed by a face of $\Gamma'$ and an edge on this face. 
The collection of all such zigzags is a $z$-orientation of ${\rm T}(\Gamma')$ with the following properties: 
all edges of $\Gamma'$ are of type II and every $v_F$ is a vertex of type I. 
This implies that ${\rm T}(\Gamma')$ satisfies the conditions (1) and (2) which gives the claim.
\end{proof}
Note that the second part of Proposition \ref{theorem1} will be used to prove the implication (3) $\Rightarrow$ (1).

\section{Structure of triangulations with faces of type I}
In this section, we describe some structural properties of $z$-oriented triangulations with faces of type I. 
As immediately consequence we obtain the implication (3) $\Rightarrow$ (1).

As above, we suppose that $(\Gamma, \tau)$ is a $z$-oriented triangulation of $M$, where all faces are of type I. 
As above, we denote by $\Gamma_{II}$ the subgraph of $\Gamma$ consisting of all vertices and all edges of type II.
From the previous section it follows that if zigzags of $(\Gamma, \tau)$ are homogeneous, then connected components of $M\setminus{\Gamma_{II}}$ are homeomorphic to an open $2$-dimensional disk. 
Now, we describe the general case.

\begin{theorem}\label{theorem2} The following assertions are fullfiled:
\begin{enumerate}
\item[{\rm (1)}] Connected components of $M\setminus{\Gamma_{II}}$ are homeomorphic to an open $2$-dimensio\-nal disk, an open M\"obius strip or an open cylinder.
\item[{\rm (2)}] A connected component of $M\setminus{\Gamma_{II}}$ contains a vertex of type I if and only if it is an open $2$-dimensional disk; such a vertex of type I is unique.
\end{enumerate}
\end{theorem}

\begin{proof}
Consider two distinct edges $e_0$ and $e_1$ of type I contained in a certain face $F_1$. 
There is precisely one face containing $e_1$ and distinct from $F_1$. 
Denote this face by $F_2$ and wright $e_2$ for the other edge of type I on $F_2$.
Recursively, we construct sequences of edges $\{e_i\}_{i\in\mathbb{N}\cup\{{0}\}}$ and faces $\{F_i\}_{i\in\mathbb{N}}$ such that $e_{i-1}$ is the common edge of $F_{i-1}, F_i$ for every $i\in\mathbb{N}$.
For any pair of the faces  $F_{i-1}, F_i$ we distinguish the following two cases presented on Fig. 5 is realized.
In the first case, the edges of type II of $F_{i-1}$ and $F_{i}$ have a common vertex (Fig. 5(a)).
In the second case (Fig. 5(b)), the edges of type II are disjoint.

\begin{center}
\begin{tikzpicture}[scale=1.2, xshift=0cm]
\begin{scope}
\draw[fill=black] (0,0) circle (1.5pt);

\draw[fill=black] (-45:2cm) circle (1.5pt);
\draw[fill=black] (-90:2cm) circle (1.5pt);
\draw[fill=black] (-135:2cm) circle (1.5pt);

\draw[thick, line width=2pt] (-45:2cm) -- (-90:2cm) -- (-135:2cm);

\draw [thick, decoration={markings,
mark=at position 0.49 with {\arrow[scale=2,>=stealth]{>}},
mark=at position 0.61 with {\arrow[scale=2,>=stealth]{>}}},
postaction={decorate}] (-135:2cm) -- (-90:2cm);

\draw [thick, decoration={markings,
mark=at position 0.49 with {\arrow[scale=2,>=stealth]{>}},
mark=at position 0.61 with {\arrow[scale=2,>=stealth]{>}}},
postaction={decorate}] (-90:2cm) -- (-45:2cm);

\draw[thick, line width=1pt, dashed] (-45:2cm) -- (0,0);
\draw[thick, line width=1pt, dashed] (-90:2cm) -- (0,0);
\draw[thick, line width=1pt, dashed] (-135:2cm) -- (0,0);

\node at (-152:1cm) {$e_{i-2}$};
\node at (-107:1cm) {$e_{i-1}$};
\node at (-35:1cm) {$e_{i}$};

\node at (-110:1.4cm) {$F_{i-1}$};
\node at (-68:1.4cm) {$F_{i}$};

\node at (-90:2.35cm) {$(a)$};
\end{scope}

\begin{scope}[xshift=4cm]
\draw[fill=black] (0,0) circle (1.5pt);

\draw[fill=black] (-90:2cm) circle (1.5pt);
\draw[fill=black] (-135:2cm) circle (1.5pt);

\draw[thick, line width=2pt] (-90:2cm) -- (-135:2cm);

\draw [thick, decoration={markings,
mark=at position 0.49 with {\arrow[scale=2,>=stealth]{>}},
mark=at position 0.61 with {\arrow[scale=2,>=stealth]{>}}},
postaction={decorate}] (-135:2cm) -- (-90:2cm);

\draw[thick, line width=1pt, dashed] (-90:2cm) -- (0,0);
\draw[thick, line width=1pt, dashed] (-135:2cm) -- (0,0);

\begin{scope}[yshift=-2cm]
\draw[fill=black] (45:2cm) circle (1.5pt);%
\draw[thick, line width=2pt] (45:2cm) -- (90:2cm);%

\draw [thick, decoration={markings,
mark=at position 0.49 with {\arrow[scale=2,>=stealth]{>}},
mark=at position 0.61 with {\arrow[scale=2,>=stealth]{>}}},
postaction={decorate}] (90:2cm) -- (45:2cm);%

\draw[thick, line width=1pt, dashed] (45:2cm) -- (0,0);%
\draw[thick, line width=1pt, dashed] (90:2cm) -- (0,0);%

\node at (68:1.2cm) {$F_{i}$};
\end{scope}

\node at (-152:1cm) {$e_{i-2}$};
\node at (-107:1cm) {$e_{i-1}$};
\node at (-8:0.8cm) {$e_{i}$};

\node at (-110:1.4cm) {$F_{i-1}$};

\node at (-90:2.35cm) {$(b)$};
\end{scope}

\end{tikzpicture}
\captionof{figure}{ }
\end{center}
Let $n$ be the smallest natural number such that $e_n=e_0$ (such a number exists by finiteness).
Therefore, the above sequences can be considered as cyclic sequences
$\{e_i\}_{i=1}^{n}$ and $\{F_i\}_{i=1}^{n}$.
The union $\mathcal{F}=\bigcup\limits_{i=1}^{n} F_{i}$ will be called a {\it component} of $(\Gamma, \tau)$. 
The boundary of $\mathcal{F}$ consists of (not necessarily all) edges of type II belonging to faces $F_i$.

Denote by $e_i^{II}$ the edge of type II belonging to $F_i$.
We take $n$ disjoint closed triangles $T_1, T_2,\dots,T_n$.
For any $i=1,2,\dots,n$ there is a homeomorphism $h_i:F_i\to T_i$ transferring any vertex and any edge of $F_i$ to a vertex and an edge of $T_i$, respectively.
We identify $h_i(e_i)$ and $h_{i+1}(e_i)$ for any $i$ in such the way that for every vertex $v$ of $e_i$ the vertices $h_{i}(v)$ and $h_{i+1}(v)$ are identified.
We get a $2$-dimensional surface $\mathcal{T}$ with boundary.
The boundary of $\mathcal{T}$ is the union of the images of all edges of type II, i.e. $\partial\mathcal{T}=\bigcup\limits_{i=1}^{n} h_i(e_i^{II})$.
Note that $\mathcal{F}$ is not necessarily a surface (since it is possible that for distinct $i,j$ the edges  $e_i^{II}, e_j^{II}$ have a common vertex).
The surface $\mathcal{T}$ is homeomorphic to one of connected components of $M\setminus{\Gamma_{II}}$ and $\mathcal{F}$ can be obtained from $\mathcal{T}$ by gluing of some parts of the boundary.

Suppose that $h_i(e_i)$ and $h_{i+1}(e_i)$ are identified only for $i=1,2,\dots,n-1$ (but not $h_1(e_0)$ and $h_n(e_n)$ from $T_{1}$ and $T_{n}$, respectively).
Then we get a space homeomorphic to a closed $2$-dimensional disk whose boundary contains $h_1(e_0), h_n(e_n)$.
Now, to complete the construction of $\mathcal{T}$, we have to glue $h_1(e_0)$ and $h_n(e_n)$.
Precisely one of the following possibilities is realized:
\\$\bullet$ A union of these sides is connected and by gluing of them we obtain that $\mathcal{T}$ homeomorphic to an open $2$-dimensional disk (Fig. 6(1)).
\\$\bullet$ The sides are disjoint and by identification of them we get a surface homeomorphic to an open M\"obius strip (Fig. 6(2)) or an open cylinder (Fig. 6(3)).

\begin{center}
\begin{tikzpicture}[scale=1.05, xshift=-1cm, yshift=0cm]
\begin{scope}[scale=0.8, xshift=-1.87cm, yshift=0.05cm]

\begin{scope}[rotate=-67.5]
\draw[fill=black] (0,0) circle (1.5pt);

\draw[fill=black] (-45:2cm) circle (1.5pt);
\draw[fill=black] (-90:2cm) circle (1.5pt);
\draw[fill=black] (-135:2cm) circle (1.5pt);

\draw[thick, line width=2pt] (-45:2cm) -- (-90:2cm) -- (-135:2cm);

\draw [thick, decoration={markings,
mark=at position 0.48 with {\arrow[scale=2,>=stealth]{>}},
mark=at position 0.62 with {\arrow[scale=2,>=stealth]{>}}},
postaction={decorate}] (-135:2cm) -- (-90:2cm);

\draw [thick, decoration={markings,
mark=at position 0.48 with {\arrow[scale=2,>=stealth]{>}},
mark=at position 0.62 with {\arrow[scale=2,>=stealth]{>}}},
postaction={decorate}] (-90:2cm) -- (-45:2cm);

\draw [thick, dashed, decoration={markings,
mark=at position 0.5 with {\arrow[scale=2]{<}}},
postaction={decorate}] (-135:2cm) -- (0,0);

\draw[thick, line width=1pt, dashed] (-45:2cm) -- (0,0);
\draw[thick, line width=1pt, dashed] (-90:2cm) -- (0,0);

\node at (-162:1.1cm) {$h_1(e_0)$};
\end{scope}
\begin{scope}[rotate=67.5]
\draw[fill=black] (0,0) circle (1.5pt);

\draw[fill=black] (-45:2cm) circle (1.5pt);
\draw[fill=black] (-90:2cm) circle (1.5pt);
\draw[fill=black] (-135:2cm) circle (1.5pt);

\draw[thick, line width=2pt] (-45:2cm) -- (-90:2cm) -- (-135:2cm);

\draw [thick, decoration={markings,
mark=at position 0.48 with {\arrow[scale=2,>=stealth]{>}},
mark=at position 0.62 with {\arrow[scale=2,>=stealth]{>}}},
postaction={decorate}] (-135:2cm) -- (-90:2cm);

\draw [thick, decoration={markings,
mark=at position 0.48 with {\arrow[scale=2,>=stealth]{>}},
mark=at position 0.62 with {\arrow[scale=2,>=stealth]{>}}},
postaction={decorate}] (-90:2cm) -- (-45:2cm);

\draw [thick, dashed, decoration={markings,
mark=at position 0.5 with {\arrow[scale=2]{<}}},
postaction={decorate}] (-45:2cm) -- (0,0);

\draw[thick, line width=1pt, dashed] (-45:2cm) -- (0,0);
\draw[thick, line width=1pt, dashed] (-90:2cm) -- (0,0);
\draw[thick, line width=1pt, dashed] (-135:2cm) -- (0,0);

\node at (-18:1.02cm) {$h_n(e_n)$};

\end{scope}
\draw[thick, dashed, line width=2pt] (-112.5:2cm) -- (-67.5:2cm);
\node at (-90:1.45cm) {\Huge{$.\hspace{0.05cm}.\hspace{0.05cm}.$}};
\node at (-90:2.35cm) {$(1)$};
\end{scope}

\begin{scope}[xshift=-5cm, yshift=-2.3cm]

\draw[fill=black] (0,0) circle (1.5pt);
\draw[fill=black] (1.5,0) circle (1.5pt);
\draw[fill=black] (3,0) circle (1.5pt);

\draw[thick, line width=2pt] (0,0) -- (3,0);
\draw [thick, decoration={markings,
mark=at position 0.48 with {\arrow[scale=2,>=stealth]{>}},
mark=at position 0.62 with {\arrow[scale=2,>=stealth]{>}}},
postaction={decorate}] (0,0) -- (1.5,0);
\draw [thick, decoration={markings,
mark=at position 0.48 with {\arrow[scale=2,>=stealth]{>}},
mark=at position 0.62 with {\arrow[scale=2,>=stealth]{>}}},
postaction={decorate}] (1.5,0) -- (3,0);

\draw[fill=black] (-0.75,-1.34) circle (1.5pt);
\draw[fill=black] (0.25,-1.34) circle (1.5pt);
\draw[fill=black] (1.25,-1.34) circle (1.5pt);
\draw[fill=black] (2.25,-1.34) circle (1.5pt);

\draw[thick, line width=2pt] (-0.75,-1.34) -- (2.25,-1.34);
\draw [thick, decoration={markings,
mark=at position 0.50 with {\arrow[scale=2,>=stealth]{>}},
mark=at position 0.68 with {\arrow[scale=2,>=stealth]{>}}},
postaction={decorate}] (-0.75,-1.34) -- (0.25,-1.34);
\draw [thick, decoration={markings,
mark=at position 0.50 with {\arrow[scale=2,>=stealth]{>}},
mark=at position 0.68 with {\arrow[scale=2,>=stealth]{>}}},
postaction={decorate}] (0.25,-1.34) -- (1.25,-1.34);
\draw [thick, decoration={markings,
mark=at position 0.50 with {\arrow[scale=2,>=stealth]{>}},
mark=at position 0.68 with {\arrow[scale=2,>=stealth]{>}}},
postaction={decorate}] (1.25,-1.34) -- (2.25,-1.34);

\draw[thick, dashed] (0,0) -- (0.25,-1.34) -- (1.5,0) -- (1.25,-1.34);
\draw[thick, dashed] (0,0) -- (0.25,-1.34) -- (1.5,0) -- (2.25,-1.34);
\draw[thick, dashed] (3,0) -- (2.25,-1.34);

\draw [thick, dashed, decoration={markings,
mark=at position 0.5 with {\arrow[scale=2]{<}}},
postaction={decorate}] (-0.75,-1.34) -- (0,0);

\node at (-1,-0.62) {$h_1(e_0)$};

\node at (3.5,-0.8) {\Huge{$.\hspace{0.15cm}.\hspace{0.15cm}.$}};
\draw[thick, dashed, line width=2pt] (3,0) -- (4,0);
\draw[thick, dashed, line width=2pt] (2.25,-1.34) -- (4.75,-1.34);

\draw[fill=black] (4,0) circle (1.5pt);
\draw[fill=black] (5,0) circle (1.5pt);
\draw[fill=black] (6,0) circle (1.5pt);
\draw[fill=black] (7,0) circle (1.5pt);

\draw[thick, line width=2pt] (4,0) -- (7,0);
\draw [thick, decoration={markings,
mark=at position 0.50 with {\arrow[scale=2,>=stealth]{>}},
mark=at position 0.68 with {\arrow[scale=2,>=stealth]{>}}},
postaction={decorate}] (4,0) -- (5.,0);
\draw [thick, decoration={markings,
mark=at position 0.50 with {\arrow[scale=2,>=stealth]{>}},
mark=at position 0.68 with {\arrow[scale=2,>=stealth]{>}}},
postaction={decorate}] (5,0) -- (6,0);
\draw [thick, decoration={markings,
mark=at position 0.50 with {\arrow[scale=2,>=stealth]{>}},
mark=at position 0.68 with {\arrow[scale=2,>=stealth]{>}}},
postaction={decorate}] (6,0) -- (7,0);

\draw[fill=black] (4.75,-1.34) circle (1.5pt);
\draw[fill=black] (5.75,-1.34) circle (1.5pt);
\draw[fill=black] (6.75,-1.34) circle (1.5pt);
\draw[fill=black] (7.75,-1.34) circle (1.5pt);

\draw[thick, line width=2pt] (4.75,-1.34) -- (7.75,-1.34);
\draw [thick, decoration={markings,
mark=at position 0.50 with {\arrow[scale=2,>=stealth]{>}},
mark=at position 0.68 with {\arrow[scale=2,>=stealth]{>}}},
postaction={decorate}] (4.75,-1.34) -- (5.75,-1.34);
\draw [thick, decoration={markings,
mark=at position 0.50 with {\arrow[scale=2,>=stealth]{>}},
mark=at position 0.68 with {\arrow[scale=2,>=stealth]{>}}},
postaction={decorate}] (5.75,-1.34) -- (6.75,-1.34);
\draw [thick, decoration={markings,
mark=at position 0.50 with {\arrow[scale=2,>=stealth]{>}},
mark=at position 0.68 with {\arrow[scale=2,>=stealth]{>}}},
postaction={decorate}] (6.75,-1.34) -- (7.75,-1.34);

\draw[thick, dashed] (4,0) -- (4.75,-1.34) -- (5,0) -- (5.75,-1.34);
\draw[thick, dashed] (5,0) -- (6.75,-1.34);
\draw[thick, dashed] (6,0) -- (6.75,-1.34) -- (7,0);

\draw [thick, dashed, decoration={markings,
mark=at position 0.5 with {\arrow[scale=2]{<}}},
postaction={decorate}] (7,0) -- (7.75,-1.34);

\node at (8.1,-0.62) {$h_n(e_n)$};

\node at (3.5,-1.72) {$(2)$};
\end{scope}


\begin{scope}[xshift=-5cm, yshift=-4.6cm]

\draw[fill=black] (0,0) circle (1.5pt);
\draw[fill=black] (1.5,0) circle (1.5pt);
\draw[fill=black] (3,0) circle (1.5pt);

\draw[thick, line width=2pt] (0,0) -- (3,0);
\draw [thick, decoration={markings,
mark=at position 0.48 with {\arrow[scale=2,>=stealth]{>}},
mark=at position 0.62 with {\arrow[scale=2,>=stealth]{>}}},
postaction={decorate}] (0,0) -- (1.5,0);
\draw [thick, decoration={markings,
mark=at position 0.48 with {\arrow[scale=2,>=stealth]{>}},
mark=at position 0.62 with {\arrow[scale=2,>=stealth]{>}}},
postaction={decorate}] (1.5,0) -- (3,0);

\draw[fill=black] (-0.75,-1.34) circle (1.5pt);
\draw[fill=black] (0.25,-1.34) circle (1.5pt);
\draw[fill=black] (1.25,-1.34) circle (1.5pt);
\draw[fill=black] (2.25,-1.34) circle (1.5pt);

\draw[thick, line width=2pt] (-0.75,-1.34) -- (2.25,-1.34);
\draw [thick, decoration={markings,
mark=at position 0.50 with {\arrow[scale=2,>=stealth]{>}},
mark=at position 0.68 with {\arrow[scale=2,>=stealth]{>}}},
postaction={decorate}] (-0.75,-1.34) -- (0.25,-1.34);
\draw [thick, decoration={markings,
mark=at position 0.50 with {\arrow[scale=2,>=stealth]{>}},
mark=at position 0.68 with {\arrow[scale=2,>=stealth]{>}}},
postaction={decorate}] (0.25,-1.34) -- (1.25,-1.34);
\draw [thick, decoration={markings,
mark=at position 0.50 with {\arrow[scale=2,>=stealth]{>}},
mark=at position 0.68 with {\arrow[scale=2,>=stealth]{>}}},
postaction={decorate}] (1.25,-1.34) -- (2.25,-1.34);

\draw[thick, dashed] (0,0) -- (0.25,-1.34) -- (1.5,0) -- (1.25,-1.34);
\draw[thick, dashed] (0,0) -- (0.25,-1.34) -- (1.5,0) -- (2.25,-1.34);
\draw[thick, dashed] (3,0) -- (2.25,-1.34);

\draw [thick, dashed, decoration={markings,
mark=at position 0.5 with {\arrow[scale=2]{<}}},
postaction={decorate}] (-0.75,-1.34) -- (0,0);

\node at (-1,-0.62) {$h_1(e_0)$};

\node at (3.5,-0.8) {\Huge{$.\hspace{0.15cm}.\hspace{0.15cm}.$}};
\draw[thick, dashed, line width=2pt] (3,0) -- (4,0);
\draw[thick, dashed, line width=2pt] (2.25,-1.34) -- (4.75,-1.34);

\draw[fill=black] (4,0) circle (1.5pt);
\draw[fill=black] (5,0) circle (1.5pt);
\draw[fill=black] (6,0) circle (1.5pt);
\draw[fill=black] (7,0) circle (1.5pt);

\draw[thick, line width=2pt] (4,0) -- (7,0);
\draw [thick, decoration={markings,
mark=at position 0.50 with {\arrow[scale=2,>=stealth]{>}},
mark=at position 0.68 with {\arrow[scale=2,>=stealth]{>}}},
postaction={decorate}] (4,0) -- (5.,0);
\draw [thick, decoration={markings,
mark=at position 0.50 with {\arrow[scale=2,>=stealth]{>}},
mark=at position 0.68 with {\arrow[scale=2,>=stealth]{>}}},
postaction={decorate}] (5,0) -- (6,0);
\draw [thick, decoration={markings,
mark=at position 0.50 with {\arrow[scale=2,>=stealth]{>}},
mark=at position 0.68 with {\arrow[scale=2,>=stealth]{>}}},
postaction={decorate}] (6,0) -- (7,0);

\draw[fill=black] (4.75,-1.34) circle (1.5pt);
\draw[fill=black] (5.75,-1.34) circle (1.5pt);
\draw[fill=black] (6.75,-1.34) circle (1.5pt);
\draw[fill=black] (7.75,-1.34) circle (1.5pt);

\draw[thick, line width=2pt] (4.75,-1.34) -- (7.75,-1.34);
\draw [thick, decoration={markings,
mark=at position 0.50 with {\arrow[scale=2,>=stealth]{>}},
mark=at position 0.68 with {\arrow[scale=2,>=stealth]{>}}},
postaction={decorate}] (4.75,-1.34) -- (5.75,-1.34);
\draw [thick, decoration={markings,
mark=at position 0.50 with {\arrow[scale=2,>=stealth]{>}},
mark=at position 0.68 with {\arrow[scale=2,>=stealth]{>}}},
postaction={decorate}] (5.75,-1.34) -- (6.75,-1.34);
\draw [thick, decoration={markings,
mark=at position 0.50 with {\arrow[scale=2,>=stealth]{>}},
mark=at position 0.68 with {\arrow[scale=2,>=stealth]{>}}},
postaction={decorate}] (6.75,-1.34) -- (7.75,-1.34);

\draw[thick, dashed] (4,0) -- (4.75,-1.34) -- (5,0) -- (5.75,-1.34);
\draw[thick, dashed] (5,0) -- (6.75,-1.34);
\draw[thick, dashed] (6,0) -- (6.75,-1.34) -- (7,0);

\draw [thick, dashed, decoration={markings,
mark=at position 0.5 with {\arrow[scale=2]{<}}},
postaction={decorate}] (7.75,-1.34) -- (7,0);

\node at (8.1,-0.62) {$h_n(e_n)$};

\node at (3.5,-1.72) {$(3)$};
\end{scope}

\end{tikzpicture}
\captionof{figure}{ }
\end{center}
Let $v_i$ be the vetrex of $T_i$ corresponding to the vertex of $F_i$ not belonging to the edge $e_i^{II}$.
In the first case, the images of edges of type I have the common vertex which is the image of all $h_i(v_i)$;
it is clear that this vertex corresponds to the vertex of type I from $\mathcal{F}$, see Fig. 6(1).
In the remaining cases, any vertex $h_i(v_i)$ is contained in the boundary of $\mathcal{T}$ and correspond to a certain vertex of $\Gamma_{II}$ (see Fig. 6(2) and 7(3)).
So, we obtained the statements (1) and (2).
\end{proof}

If a connected component of $M\setminus\Gamma_{II}$ is homeomorphic to an open $2$-dimensional disk, then the corresponding component of $(\Gamma,\tau)$ is homeomorphic to a closed $2$-dimensional disk
(if this component has some identifications at the boundary, then the vertex of type I in this component is joined by a double edge to a certain vertex at the boundary which is impossible, since we work with embeddings of simple graphs).

\begin{proof}[Proof of (3) $\!\Rightarrow\!$ (1) in Theorem \ref{newth1}]
Assume that each connected component of ${M\!\setminus\!\Gamma_{\!I\!I}}$ is a disk. 
By the above remark, $\Gamma_{II}$ is a closed 2-cell embedding. 
Lemma \ref{lemma1} show that this is an embedding of simple Eulerian digraph. 
The second part of Theorem \ref{theorem2} states that each disk contains a unique vertex of type I; 
as in the proof of Theorem \ref{theorem2} we establish that its boundary is an oriented cycle. 
We have $\Gamma={\rm T}(\Gamma_{II})$ and the second part of Proposition \ref{theorem1} gives the claim.
\end{proof}

The following three examples show that all possibilities for connected components of $M\setminus\Gamma_{II}$ are realized.
\begin{exmp}\label{ex9}{\rm
Consider the following triangulation $\Gamma$ of a torus $\mathbb{T}=\mathbb{S}^{1}\times\mathbb{S}^{1}$ (see Fig. 7).
It contains precisely three zigzags (up to reversing):
\begin{enumerate}
\item[$\bullet$] $e_0, e'_0, e_1, e'_1, e_2, e'_2, e_3, e'_3, e_4 e'_4$,
\item[$\bullet$] $0, e'_0, 2, e'_2, 4, e'_4, 1, e'_1, 3, e'_3,$
\item[$\bullet$] $0, e'_4, 4, e'_3, 3, e'_2, 2, e'_1, 1, e'_0.$
\end{enumerate} 
In the corresponding $z$-orientation $\tau$ all faces are of type I.
The subgraph $\Gamma_{II}$ is the simple $5$-cycle, but $\mathbb{T}\setminus\Gamma_{II}$ is connected and homeomorphic to the open cylinder $(-1,1)\times\mathbb{S}^{1}$.
Moreover, there is only one component of $(\Gamma,\tau)$ homeomorphic to the entire surface $\mathbb{T}$. 
\begin{center}
\begin{tikzpicture}[scale=0.8, xshift=0cm]

\draw[fill=black] (90:1.5cm) circle (1.5pt);
\draw[fill=black] (162:1.5cm) circle (1.5pt);
\draw[fill=black] (234:1.5cm) circle (1.5pt);
\draw[fill=black] (306:1.5cm) circle (1.5pt);
\draw[fill=black] (378:1.5cm) circle (1.5pt);

\draw[fill=black] (90:3cm) circle (1.5pt);
\draw[fill=black] (162:3cm) circle (1.5pt);
\draw[fill=black] (234:3cm) circle (1.5pt);
\draw[fill=black] (306:3cm) circle (1.5pt);
\draw[fill=black] (378:3cm) circle (1.5pt);

\draw[thick, line width=2pt] (90:1.5cm) -- (162:1.5cm) -- (234:1.5cm) -- (306:1.5cm) -- (378:1.5cm) -- cycle;
\draw[thick, line width=2pt] (90:3cm) -- (162:3cm) -- (234:3cm) -- (306:3cm) -- (378:3cm) -- cycle;

\draw[thick, line width=1pt, dashed] (234:3cm) -- (234:1.5cm) -- (306:3cm) -- (306:1.5cm) -- (378:3cm) -- (378:1.5cm) -- (90:3cm) -- (90:1.5cm) -- (162:3cm) -- (162:1.5cm) -- cycle;

\draw [thick, decoration={markings,
mark=at position 0.49 with {\arrow[scale=2,>=stealth]{>}},
mark=at position 0.61 with {\arrow[scale=2,>=stealth]{>}}},
postaction={decorate}] (90:1.5cm) -- (162:1.5cm);
\draw [thick, decoration={markings,
mark=at position 0.49 with {\arrow[scale=2,>=stealth]{>}},
mark=at position 0.61 with {\arrow[scale=2,>=stealth]{>}}},
postaction={decorate}] (162:1.5cm) -- (234:1.5cm);
\draw [thick, decoration={markings,
mark=at position 0.49 with {\arrow[scale=2,>=stealth]{>}},
mark=at position 0.61 with {\arrow[scale=2,>=stealth]{>}}},
postaction={decorate}] (234:1.5cm) -- (306:1.5cm);
\draw [thick, decoration={markings,
mark=at position 0.49 with {\arrow[scale=2,>=stealth]{>}},
mark=at position 0.61 with {\arrow[scale=2,>=stealth]{>}}},
postaction={decorate}] (306:1.5cm) -- (378:1.5cm);
\draw [thick, decoration={markings,
mark=at position 0.49 with {\arrow[scale=2,>=stealth]{>}},
mark=at position 0.61 with {\arrow[scale=2,>=stealth]{>}}},
postaction={decorate}] (378:1.5cm) -- (90:1.5cm);

\draw [thick, decoration={markings,
mark=at position 0.52 with {\arrow[scale=2,>=stealth]{>}},
mark=at position 0.58 with {\arrow[scale=2,>=stealth]{>}}},
postaction={decorate}] (90:3cm) -- (162:3cm);
\draw [thick, decoration={markings,
mark=at position 0.52 with {\arrow[scale=2,>=stealth]{>}},
mark=at position 0.58 with {\arrow[scale=2,>=stealth]{>}}},
postaction={decorate}] (162:3cm) -- (234:3cm);
\draw [thick, decoration={markings,
mark=at position 0.52 with {\arrow[scale=2,>=stealth]{>}},
mark=at position 0.58 with {\arrow[scale=2,>=stealth]{>}}},
postaction={decorate}] (234:3cm) -- (306:3cm);
\draw [thick, decoration={markings,
mark=at position 0.52 with {\arrow[scale=2,>=stealth]{>}},
mark=at position 0.58 with {\arrow[scale=2,>=stealth]{>}}},
postaction={decorate}] (306:3cm) -- (378:3cm);
\draw [thick, decoration={markings,
mark=at position 0.52 with {\arrow[scale=2,>=stealth]{>}},
mark=at position 0.58 with {\arrow[scale=2,>=stealth]{>}}},
postaction={decorate}] (378:3cm) -- (90:3cm);

\node at (-90:0.85cm) {$2$};
\node at (-18:0.85cm) {$3$};
\node at (54:0.85cm) {$4$};
\node at (126:0.85cm) {$0$};
\node at (198:0.85cm) {$1$};

\node at (-90:2.8cm) {$0$};
\node at (-18:2.8cm) {$1$};
\node at (54:2.8cm) {$2$};
\node at (126:2.8cm) {$3$};
\node at (198:2.8cm) {$4$};

\node at (-90:1.9cm) {$e'_0$};
\node at (-18:1.9cm) {$e'_1$};
\node at (54:1.9cm) {$e'_2$};
\node at (126:1.9cm) {$e'_3$};
\node at (198:1.9cm) {$e'_4$};

\node at (97:2.25cm) {$e_3$}; 
\node at (169:2.25cm) {$e_4$};
\node at (241:2.25cm) {$e_0$};
\node at (313:2.25cm) {$e_1$};
\node at (385:2.25cm) {$e_2$};

\end{tikzpicture}
\captionof{figure}{ }
\end{center}
This example can be generalized as follows.
Let $n,k\in\mathbb{N}$ and let $\Gamma$ be the triangulation of a torus presented on the figure below (each horizontal edge denoted by $l$ is identified with the edge labeled by $l+k\:(\text{mod }n$).

\begin{center}
\begin{tikzpicture}[scale=0.8, xshift=0cm]

\draw[fill=black] (0,0) circle (1.5pt);
\draw[fill=black] (2,0) circle (1.5pt);
\draw[fill=black] (4,0) circle (1.5pt);
\draw[fill=black] (6,0) circle (1.5pt);
\draw[fill=black] (8,0) circle (1.5pt);

\draw[fill=black] (10,0) circle (1.5pt);
\draw[fill=black] (12,0) circle (1.5pt);

\draw[fill=black] (0,2) circle (1.5pt);
\draw[fill=black] (2,2) circle (1.5pt);
\draw[fill=black] (4,2) circle (1.5pt);
\draw[fill=black] (6,2) circle (1.5pt);
\draw[fill=black] (8,2) circle (1.5pt);

\draw[fill=black] (10,2) circle (1.5pt);
\draw[fill=black] (12,2) circle (1.5pt);

\draw[thick, line width=2pt] (0,0) -- (8,0);
\draw[thick, line width=2pt] (10,0) -- (12,0);

\draw[thick, line width=2pt] (0,2) -- (8,2);
\draw[thick, line width=2pt] (10,2) -- (12,2);

\draw[thick, line width=1pt, dashed] (0,0) -- (0,2) -- (2,0) -- (2,2) -- (4,0) -- (4,2) -- (6,0) -- (6,2) -- (8,0) -- (8,2);
\draw[thick, line width=1pt, dashed] (10,0) -- (10,2) -- (12,0) -- (12,2);

\draw [thick, decoration={markings,
mark=at position 0.49 with {\arrow[scale=2,>=stealth]{>}},
mark=at position 0.61 with {\arrow[scale=2,>=stealth]{>}}},
postaction={decorate}] (0,0) -- (2,0);
\draw [thick, decoration={markings,
mark=at position 0.49 with {\arrow[scale=2,>=stealth]{>}},
mark=at position 0.61 with {\arrow[scale=2,>=stealth]{>}}},
postaction={decorate}] (2,0) -- (4,0);
\draw [thick, decoration={markings,
mark=at position 0.49 with {\arrow[scale=2,>=stealth]{>}},
mark=at position 0.61 with {\arrow[scale=2,>=stealth]{>}}},
postaction={decorate}] (4,0) -- (6,0);
\draw [thick, decoration={markings,
mark=at position 0.49 with {\arrow[scale=2,>=stealth]{>}},
mark=at position 0.61 with {\arrow[scale=2,>=stealth]{>}}},
postaction={decorate}] (6,0) -- (8,0);

\draw [thick, decoration={markings,
mark=at position 0.49 with {\arrow[scale=2,>=stealth]{>}},
mark=at position 0.61 with {\arrow[scale=2,>=stealth]{>}}},
postaction={decorate}] (10,0) -- (12,0);

\draw [thick, decoration={markings,
mark=at position 0.49 with {\arrow[scale=2,>=stealth]{>}},
mark=at position 0.61 with {\arrow[scale=2,>=stealth]{>}}},
postaction={decorate}] (0,2) -- (2,2);
\draw [thick, decoration={markings,
mark=at position 0.49 with {\arrow[scale=2,>=stealth]{>}},
mark=at position 0.61 with {\arrow[scale=2,>=stealth]{>}}},
postaction={decorate}] (2,2) -- (4,2);
\draw [thick, decoration={markings,
mark=at position 0.49 with {\arrow[scale=2,>=stealth]{>}},
mark=at position 0.61 with {\arrow[scale=2,>=stealth]{>}}},
postaction={decorate}] (4,2) -- (6,2);
\draw [thick, decoration={markings,
mark=at position 0.49 with {\arrow[scale=2,>=stealth]{>}},
mark=at position 0.61 with {\arrow[scale=2,>=stealth]{>}}},
postaction={decorate}] (6,2) -- (8,2);

\draw [thick, decoration={markings,
mark=at position 0.49 with {\arrow[scale=2,>=stealth]{>}},
mark=at position 0.61 with {\arrow[scale=2,>=stealth]{>}}},
postaction={decorate}] (10,2) -- (12,2);

\draw[thick, dashed, line width=2pt] (8,0) -- (10,0);
\draw[thick, dashed, line width=2pt] (8,2) -- (10,2);

\node at (9,1) {\Huge{$.\hspace{0.15cm}.\hspace{0.15cm}.$}};

\draw [thick, dashed, decoration={markings,
mark=at position 0.5 with {\arrow[scale=2]{<}}},
postaction={decorate}] (0,2) -- (0,0);
\draw [thick, dashed, decoration={markings,
mark=at position 0.5 with {\arrow[scale=2]{<}}},
postaction={decorate}] (12,2) -- (12,0);

\node at (1,-0.4) {$0$};
\node at (3,-0.4) {$1$};
\node at (5,-0.4) {$2$};
\node at (7,-0.4) {$3$};
\node at (11,-0.4) {$n\!-\!1$};

\node at (1,2.4) {$k$};
\node at (3,2.4) {$1\!+\!k$};
\node at (5,2.4) {$2\!+\!k$};
\node at (7,2.4) {$3\!+\!k$};
\node at (11,2.4) {$n\!-\!1\!+\!k$};

\node at (1.2,1.2) {$e'_0$};
\node at (3.2,1.2) {$e'_1$};
\node at (5.2,1.2) {$e'_2$};
\node at (7.2,1.2) {$e'_3$};
\node at (11.38,1.2) {$e'_{n-1}$};

\node at (-0.35,0.9) {$e_0$};
\node at (1.7,0.9) {$e_1$};
\node at (3.7,0.9) {$e_2$};
\node at (5.7,0.9) {$e_3$};
\node at (10.5,0.9) {$e_{n-1}$};
\node at (12.4,0.9) {$e_0$};

\end{tikzpicture}
\captionof{figure}{ }
\end{center}
It is necessarily to assume that $n\geq 5$ and $2\leq k\leq n-3$ to avoid multiple edges and loops.
This triangulation has the following three types of zigzags (up to reversing):
\begin{enumerate}
\item[$\bullet$] $e_0, e'_0, e_1, e'_1, e_2, e'_2, e_3, e'_3, \dots, e_{n-1}, e'_{n-1}$,
\item[$\bullet$] zigzags of type $l, e'_{l}, l+k, e'_{l+k}, l+2k, e'_{l+2k}, \dots$
\item[$\bullet$] zigzags of type $l, e'_{l+1-k}, l+1-k, e'_{l+2(1-k)}, l+2(1-k), e'_{l+3(1-k)}, \dots$
\end{enumerate} 
As above, we come to a $z$-oriented triangulation with the faces of type I such that
$\Gamma_{II}$ is the simple $n$-cycle and 
$\mathbb{T}\setminus\Gamma_{II}$ is homeomorphic to an open cylinder.
}\end{exmp}

\begin{exmp}\label{ex10}{\rm
Let $n\in\mathbb{N}$ and let $\Gamma$ be the triangulation of a real projective plane obtained by gluing of boundaries of a M\"obius strip and a closed $2$-dimensional disk (see Fig. 9).
According to the corresponding $z$-orientation all faces are of type I and the graph $\Gamma_{II}$ consists of all edges marked by the double arrows and their vertices.
Then $\mathbb{R}\text{P}^2\setminus\Gamma_{II}$ has two connected components.
One of them is homeomorphic to an open $2$-dimensional disk and the remaining to an open M\"obius strip.
\begin{center}
\begin{tikzpicture}[scale=1, xshift=0cm]
\begin{scope}[scale=0.8]
\draw[fill=black] (0,0) circle (1.5pt);
\draw[fill=black] (2,0) circle (1.5pt);
\draw[fill=black] (4,0) circle (1.5pt);
\draw[fill=black] (6,0) circle (1.5pt);
\draw[fill=black] (8,0) circle (1.5pt);
\draw[fill=black] (10,0) circle (1.5pt);
\draw[fill=black] (12,0) circle (1.5pt);

\draw[fill=black] (0,2) circle (1.5pt);
\draw[fill=black] (2,2) circle (1.5pt);
\draw[fill=black] (4,2) circle (1.5pt);
\draw[fill=black] (6,2) circle (1.5pt);
\draw[fill=black] (8,2) circle (1.5pt);

\draw[fill=black] (10,2) circle (1.5pt);
\draw[fill=black] (12,2) circle (1.5pt);

\draw[thick, line width=2pt] (0,0) -- (8,0);
\draw[thick, line width=2pt] (10,0) -- (12,0);

\draw[thick, line width=2pt] (0,2) -- (8,2);
\draw[thick, line width=2pt] (10,2) -- (12,2);

\draw[thick, line width=1pt, dashed] (0,0) -- (0,2) -- (2,0) -- (2,2) -- (4,0) -- (4,2) -- (6,0) -- (6,2) -- (8,0) -- (8,2);
\draw[thick, line width=1pt, dashed] (10,0) -- (10,2) -- (12,0);

\draw [thick, decoration={markings,
mark=at position 0.49 with {\arrow[scale=2,>=stealth]{>}},
mark=at position 0.61 with {\arrow[scale=2,>=stealth]{>}}},
postaction={decorate}] (0,0) -- (2,0);
\draw [thick, decoration={markings,
mark=at position 0.49 with {\arrow[scale=2,>=stealth]{>}},
mark=at position 0.61 with {\arrow[scale=2,>=stealth]{>}}},
postaction={decorate}] (2,0) -- (4,0);
\draw [thick, decoration={markings,
mark=at position 0.49 with {\arrow[scale=2,>=stealth]{>}},
mark=at position 0.61 with {\arrow[scale=2,>=stealth]{>}}},
postaction={decorate}] (4,0) -- (6,0);
\draw [thick, decoration={markings,
mark=at position 0.49 with {\arrow[scale=2,>=stealth]{>}},
mark=at position 0.61 with {\arrow[scale=2,>=stealth]{>}}},
postaction={decorate}] (6,0) -- (8,0);

\draw [thick, decoration={markings,
mark=at position 0.49 with {\arrow[scale=2,>=stealth]{>}},
mark=at position 0.61 with {\arrow[scale=2,>=stealth]{>}}},
postaction={decorate}] (10,0) -- (12,0);

\draw [thick, decoration={markings,
mark=at position 0.49 with {\arrow[scale=2,>=stealth]{>}},
mark=at position 0.61 with {\arrow[scale=2,>=stealth]{>}}},
postaction={decorate}] (0,2) -- (2,2);
\draw [thick, decoration={markings,
mark=at position 0.49 with {\arrow[scale=2,>=stealth]{>}},
mark=at position 0.61 with {\arrow[scale=2,>=stealth]{>}}},
postaction={decorate}] (2,2) -- (4,2);
\draw [thick, decoration={markings,
mark=at position 0.49 with {\arrow[scale=2,>=stealth]{>}},
mark=at position 0.61 with {\arrow[scale=2,>=stealth]{>}}},
postaction={decorate}] (4,2) -- (6,2);
\draw [thick, decoration={markings,
mark=at position 0.49 with {\arrow[scale=2,>=stealth]{>}},
mark=at position 0.61 with {\arrow[scale=2,>=stealth]{>}}},
postaction={decorate}] (6,2) -- (8,2);

\draw [thick, decoration={markings,
mark=at position 0.49 with {\arrow[scale=2,>=stealth]{>}},
mark=at position 0.61 with {\arrow[scale=2,>=stealth]{>}}},
postaction={decorate}] (10,2) -- (12,2);

\draw[thick, dashed, line width=2pt] (8,0) -- (10,0);
\draw[thick, dashed, line width=2pt] (8,2) -- (10,2);

\node at (9,1) {\Huge{$.\hspace{0.15cm}.\hspace{0.15cm}.$}};

\draw [thick, dashed, decoration={markings,
mark=at position 0.5 with {\arrow[scale=2]{<}}},
postaction={decorate}] (0,0) -- (0,2);
\draw [thick, dashed, decoration={markings,
mark=at position 0.5 with {\arrow[scale=2]{<}}},
postaction={decorate}] (12,2) -- (12,0);

\node at (1,2.4) {$n\!+\!1$};
\node at (3,2.4) {$n\!+\!2$};
\node at (5,2.4) {$n\!+\!3$};
\node at (7,2.4) {$n\!+\!4$};
\node at (11,2.4) {$2n$};

\node at (-0.3,0.9) {$e$};
\node at (12.25,0.9) {$e$};

\node at (1,-0.4) {$1$};
\node at (3,-0.4) {$2$};
\node at (5,-0.4) {$3$};
\node at (7,-0.4) {$4$};
\node at (11,-0.4) {$n$};

\end{scope}


\begin{scope}[xshift=4.83cm, yshift=-2.8cm, rotate=-60]

\draw[fill=black] (0,0) circle (1.5pt);

\draw[fill=black] (15:2cm) circle (1.5pt);
\draw[fill=black] (45:2cm) circle (1.5pt);
\draw[fill=black] (75:2cm) circle (1.5pt);
\draw[fill=black] (105:2cm) circle (1.5pt);
\draw[fill=black] (135:2cm) circle (1.5pt);
\draw[fill=black] (165:2cm) circle (1.5pt);
\draw[fill=black] (195:2cm) circle (1.5pt);
\draw[fill=black] (225:2cm) circle (1.5pt);
\draw[fill=black] (255:2cm) circle (1.5pt);
\draw[fill=black] (285:2cm) circle (1.5pt);
\draw[fill=black] (315:2cm) circle (1.5pt);
\draw[fill=black] (345:2cm) circle (1.5pt);

\draw [thick, decoration={markings,
mark=at position 0.54 with {\arrow[scale=2,>=stealth]{<}},
mark=at position 0.7 with {\arrow[scale=2,>=stealth]{<}}},
postaction={decorate}] (15:2cm) -- (45:2cm);
\draw [thick, decoration={markings,
mark=at position 0.54 with {\arrow[scale=2,>=stealth]{<}},
mark=at position 0.7 with {\arrow[scale=2,>=stealth]{<}}},
postaction={decorate}] (45:2cm) -- (75:2cm);
\draw [thick, decoration={markings,
mark=at position 0.54 with {\arrow[scale=2,>=stealth]{<}},
mark=at position 0.7 with {\arrow[scale=2,>=stealth]{<}}},
postaction={decorate}] (75:2cm) -- (105:2cm);
\draw [thick, decoration={markings,
mark=at position 0.54 with {\arrow[scale=2,>=stealth]{<}},
mark=at position 0.7 with {\arrow[scale=2,>=stealth]{<}}},
postaction={decorate}] (105:2cm) -- (135:2cm);
\draw [thick, decoration={markings,
mark=at position 0.54 with {\arrow[scale=2,>=stealth]{<}},
mark=at position 0.7 with {\arrow[scale=2,>=stealth]{<}}},
postaction={decorate}] (165:2cm) -- (195:2cm);
\draw [thick, decoration={markings,
mark=at position 0.54 with {\arrow[scale=2,>=stealth]{<}},
mark=at position 0.7 with {\arrow[scale=2,>=stealth]{<}}},
postaction={decorate}] (195:2cm) -- (225:2cm);
\draw [thick, decoration={markings,
mark=at position 0.54 with {\arrow[scale=2,>=stealth]{<}},
mark=at position 0.7 with {\arrow[scale=2,>=stealth]{<}}},
postaction={decorate}] (225:2cm) -- (255:2cm);
\draw [thick, decoration={markings,
mark=at position 0.54 with {\arrow[scale=2,>=stealth]{<}},
mark=at position 0.7 with {\arrow[scale=2,>=stealth]{<}}},
postaction={decorate}] (255:2cm) -- (285:2cm);
\draw [thick, decoration={markings,
mark=at position 0.54 with {\arrow[scale=2,>=stealth]{<}},
mark=at position 0.7 with {\arrow[scale=2,>=stealth]{<}}},
postaction={decorate}] (285:2cm) -- (315:2cm);
\draw [thick, decoration={markings,
mark=at position 0.54 with {\arrow[scale=2,>=stealth]{<}},
mark=at position 0.7 with {\arrow[scale=2,>=stealth]{<}}},
postaction={decorate}] (345:2cm) -- (15:2cm);

\draw[thick, line width=2pt] (345:2cm) -- (15:2cm) -- (45:2cm) -- (75:2cm) -- (105:2cm) -- (105:2cm) -- (135:2cm);
\draw[thick, line width=2pt] (165:2cm) -- (195:2cm) -- (225:2cm) -- (255:2cm) -- (285:2cm) -- (315:2cm);
\draw[thick, dashed, line width=2pt] (135:2cm) -- (165:2cm);
\draw[thick, dashed, line width=2pt] (315:2cm) -- (345:2cm);

\draw[thick, line width=1pt, dashed] (345:2cm) -- (0,0) -- (15:2cm);
\draw[thick, line width=1pt, dashed] (45:2cm) -- (0,0) -- (75:2cm);
\draw[thick, line width=1pt, dashed] (105:2cm) -- (0,0) -- (135:2cm);
\draw[thick, line width=1pt, dashed] (165:2cm) -- (0,0) -- (195:2cm);
\draw[thick, line width=1pt, dashed] (225:2cm) -- (0,0) -- (255:2cm);
\draw[thick, line width=1pt, dashed] (285:2cm) -- (0,0) -- (315:2cm);

\node at (0:2.3cm) {$4$};
\node at (30:2.3cm) {$3$};
\node at (60:2.3cm) {$2$};
\node at (90:2.3cm) {$1$};
\node at (120:2.3cm) {$2n$};
\node at (180:2.3cm) {$n\!+\!4$};
\node at (210:2.45cm) {$n\!+\!3$};
\node at (240:2.45cm) {$n\!+\!2$};
\node at (270:2.45cm) {$n\!+\!1$};
\node at (300:2.3cm) {$n$};

\node at (-30:1.45cm) {\Huge{$.\hspace{0.015cm}.\hspace{0.015cm}.$}};
\node at (150:1.45cm) {\Huge{$.\hspace{0.015cm}.\hspace{0.015cm}.$}};
\end{scope}

\end{tikzpicture}
\captionof{figure}{ }
\end{center}
}\end{exmp}

\begin{exmp}\label{ex11}{\rm
Suppose that $\Gamma$ is the triangulation of a sphere obtained by the gluing of the two disks whose boundaries are cycles $e_1,e_2,\dots,e_6$ (see Fig. 10). There is a $z$-orientation $\tau$ such that all faces are of type I.
Then $\mathbb{S}^2\setminus\Gamma_{II}$ has precisely four connected components: three components are homeomorphic to an open $2$-dimensional disk and the remaining to an open cylinder.
The components of $(\Gamma,\tau)$ corresponding to the first three connected components are closed $2$-dimensional disks.
The fourth component of $(\Gamma,\tau)$ is homeomorphic to a closed cylinder $\mathbb{S}^{1}\times D^1$, where two points at one of the connected components of the boundary are glued.
\begin{center}
\begin{tikzpicture}[scale=0.8, xshift=0cm]

\draw[fill=black] (0,0) circle (1.75pt);
\draw[fill=black] (4.5,0) circle (1.75pt);

\draw[fill=black] (1.5,1.5) circle (1.75pt);
\draw[fill=black] (3,1.5) circle (1.75pt);

\draw[fill=black] (2.25,3) circle (1.75pt);

\draw[fill=black] (3,4.5) circle (1.75pt);
\draw[fill=black] (1.5,4.5) circle (1.75pt);

\draw[fill=black] (4.5,6) circle (1.75pt);
\draw[fill=black] (0,6) circle (1.75pt);

\draw[fill=black] (0,3) circle (1.75pt);
\draw[fill=black] (4.5,3) circle (1.75pt);

\draw [thick, decoration={markings,
mark=at position 0.51 with {\arrow[scale=2,>=stealth]{>}},
mark=at position 0.59 with {\arrow[scale=2,>=stealth]{>}}},
postaction={decorate}] (0,0) -- (0,3);
\draw [thick, decoration={markings,
mark=at position 0.51 with {\arrow[scale=2,>=stealth]{>}},
mark=at position 0.59 with {\arrow[scale=2,>=stealth]{>}}},
postaction={decorate}] (0,3) -- (0,6);

\draw [thick, decoration={markings,
mark=at position 0.525 with {\arrow[scale=2,>=stealth]{>}},
mark=at position 0.575 with {\arrow[scale=2,>=stealth]{>}}},
postaction={decorate}] (0,6) -- (4.5,6);

\draw [thick, decoration={markings,
mark=at position 0.51 with {\arrow[scale=2,>=stealth]{>}},
mark=at position 0.59 with {\arrow[scale=2,>=stealth]{>}}},
postaction={decorate}] (4.5,6) -- (4.5,3);
\draw [thick, decoration={markings,
mark=at position 0.51 with {\arrow[scale=2,>=stealth]{>}},
mark=at position 0.59 with {\arrow[scale=2,>=stealth]{>}}},
postaction={decorate}] (4.5,3) -- (4.5,0);

\draw [thick, decoration={markings,
mark=at position 0.525 with {\arrow[scale=2,>=stealth]{>}},
mark=at position 0.575 with {\arrow[scale=2,>=stealth]{>}}},
postaction={decorate}] (4.5,0) -- (0,0);

\draw [thick, decoration={markings,
mark=at position 0.58 with {\arrow[scale=2,>=stealth]{>}},
mark=at position 0.72 with {\arrow[scale=2,>=stealth]{>}}},
postaction={decorate}] (2.25,3) -- (1.5,4.5);
\draw [thick, decoration={markings,
mark=at position 0.58 with {\arrow[scale=2,>=stealth]{>}},
mark=at position 0.72 with {\arrow[scale=2,>=stealth]{>}}},
postaction={decorate}] (1.5,4.5) -- (3,4.5);
\draw [thick, decoration={markings,
mark=at position 0.58 with {\arrow[scale=2,>=stealth]{>}},
mark=at position 0.72 with {\arrow[scale=2,>=stealth]{>}}},
postaction={decorate}] (3,4.5) -- (2.25,3);

\draw [thick, decoration={markings,
mark=at position 0.58 with {\arrow[scale=2,>=stealth]{>}},
mark=at position 0.72 with {\arrow[scale=2,>=stealth]{>}}},
postaction={decorate}] (2.25,3) -- (3,1.5);
\draw [thick, decoration={markings,
mark=at position 0.58 with {\arrow[scale=2,>=stealth]{>}},
mark=at position 0.72 with {\arrow[scale=2,>=stealth]{>}}},
postaction={decorate}] (3,1.5) -- (1.5,1.5);
\draw [thick, decoration={markings,
mark=at position 0.58 with {\arrow[scale=2,>=stealth]{>}},
mark=at position 0.72 with {\arrow[scale=2,>=stealth]{>}}},
postaction={decorate}] (1.5,1.5) -- (2.25,3);

\draw[thick, line width=2pt] (0,0) -- (4.5,0) -- (4.5,6) -- (0,6) -- (0,0);
\draw[thick, line width=2pt] (1.5,1.5) -- (3,1.5) -- (2.25,3) -- (3,4.5) -- (1.5,4.5) -- (2.25,3) -- (1.5,1.5);

\draw[thick, dashed] (1.5,4.5) -- (2.25,3.9);
\draw[thick, dashed] (3,4.5) -- (2.25,3.9);
\draw[thick, dashed] (2.25,3) -- (2.25,3.9);

\draw[thick, dashed] (1.5,1.5) -- (2.25,2.1);
\draw[thick, dashed] (3,1.5) -- (2.25,2.1);
\draw[thick, dashed] (2.25,3) -- (2.25,2.1);

\draw[thick, dashed] (0,0) -- (1.5,1.5) -- (0,3) -- (2.25,3) -- (0,6) -- (1.5,4.5) -- (4.5,6) -- (3,4.5) -- (4.5,3) -- (2.25,3) -- (4.5,0) -- (3,1.5) -- (0,0);

\draw[fill=black] (2.25,3.9) circle (1.75pt);
\draw[fill=black] (2.25,2.1) circle (1.75pt);

\node at (2.25,-0.4) {$e_4$};
\node at (-0.4,1.5) {$e_5$};
\node at (-0.4,4.5) {$e_6$};
\node at (2.25,6.35) {$e_1$};
\node at (4.9,4.5) {$e_2$};
\node at (4.9,1.5) {$e_3$};

\begin{scope}[xshift=6.5cm]
\draw[fill=black] (0,0) circle (1.75pt);
\draw[fill=black] (4.5,0) circle (1.75pt);

\draw[fill=black] (4.5,6) circle (1.75pt);
\draw[fill=black] (0,6) circle (1.75pt);

\draw[fill=black] (0,3) circle (1.75pt);
\draw[fill=black] (4.5,3) circle (1.75pt);

\draw [thick, decoration={markings,
mark=at position 0.51 with {\arrow[scale=2,>=stealth]{>}},
mark=at position 0.59 with {\arrow[scale=2,>=stealth]{>}}},
postaction={decorate}] (0,0) -- (0,3);
\draw [thick, decoration={markings,
mark=at position 0.51 with {\arrow[scale=2,>=stealth]{>}},
mark=at position 0.59 with {\arrow[scale=2,>=stealth]{>}}},
postaction={decorate}] (0,3) -- (0,6);

\draw [thick, decoration={markings,
mark=at position 0.525 with {\arrow[scale=2,>=stealth]{>}},
mark=at position 0.575 with {\arrow[scale=2,>=stealth]{>}}},
postaction={decorate}] (0,6) -- (4.5,6);

\draw [thick, decoration={markings,
mark=at position 0.51 with {\arrow[scale=2,>=stealth]{>}},
mark=at position 0.59 with {\arrow[scale=2,>=stealth]{>}}},
postaction={decorate}] (4.5,6) -- (4.5,3);
\draw [thick, decoration={markings,
mark=at position 0.51 with {\arrow[scale=2,>=stealth]{>}},
mark=at position 0.59 with {\arrow[scale=2,>=stealth]{>}}},
postaction={decorate}] (4.5,3) -- (4.5,0);

\draw [thick, decoration={markings,
mark=at position 0.525 with {\arrow[scale=2,>=stealth]{>}},
mark=at position 0.575 with {\arrow[scale=2,>=stealth]{>}}},
postaction={decorate}] (4.5,0) -- (0,0);

\draw[thick, line width=2pt] (0,0) -- (4.5,0) -- (4.5,6) -- (0,6) -- (0,0);

\draw[thick, dashed] (0,0) -- (4.5,6);
\draw[thick, dashed] (0,3) -- (4.5,3);
\draw[thick, dashed] (0,6) -- (4.5,0);

\draw[fill=black] (2.25,3) circle (1.75pt);

\node at (2.25,-0.4) {$e_4$};
\node at (-0.4,1.5) {$e_5$};
\node at (-0.4,4.5) {$e_6$};
\node at (2.25,6.35) {$e_1$};
\node at (4.9,4.5) {$e_2$};
\node at (4.9,1.5) {$e_3$};
\end{scope}

\end{tikzpicture}
\captionof{figure}{ }
\end{center}
}\end{exmp}

\section{Relations to $Z$-monodromies}

Let $F$ be a face in $\Gamma$ whose vertices are $a,b,c$. 
Consider the set $\Omega(F)$ of all oriented edges of $F$
$$\Omega(F)=\{ab,bc,ca,ac,cb,ba\},$$
where $xy$ is the edge from $x\in\{a,b,c\}$ to $y\in\{a,b,c\}$. 
If $e=xy$ then we write $yx$ by $-e$.
Denote by $D_F$ the following permutation of the set $\Omega(F)$
$$(ab,bc,ca)(ac,cb,ba).$$
The {\it $z$-monodromy} (see \cite{PT2, PT3}) of the face $F$ is the permutation $M_F$ defined as follows. 
For any $e\in \Omega(F)$ we take $e_{0}\in \Omega(F)$ such that $D_{F}(e_{0})=e$
and consider the zigzag containing the sequence $e_{0},e$. 
We define $M_{F}(e)$ as the first element of  $\Omega(F)$ contained in this zigzag after $e$.
 
By \cite[Theorem 2]{PT2}, there are the following possibilities for the $z$-monodromy $M_{F}$ and each of them is realized:
\begin{enumerate}
\item[(M1)] $M_{F}$ is identity,
\item[(M2)] $M_{F}=D_{F}$,
\item[(M3)] $M_{F}=(-e_{1},e_{2},e_{3})(-e_{3},-e_{2},e_{1})$,
\item[(M4)] $M_{F}=(e_{1},-e_{2})(e_{2},-e_{1})$, where {\rm}$e_{3}$ and $-e_{3}$ are fixed points{\rm},
\item[(M5)] $M_{F}=(D_{F})^{-1}$,
\item[(M6)] $M_{F}=(-e_{1},e_{3},e_{2})(-e_{2},-e_{3},e_{1})$,
\item[(M7)] $M_{F}=(e_{1},e_{2})(-e_{1},-e_{2})$, where {\rm}$e_{3}$ and $-e_{3}$ are fixed points{\rm}
\end{enumerate}
where $(e_{1},e_{2},e_{3})$ is one of the cycles in $D_{F}$.

Let ${\mathrm G}_{i}$ be the subgraph of the dual $\Gamma^{*}$ formed by vertices corresponding to faces in $\Gamma$ whose $z$-monodromies are of type (M$i$), two vertices of ${\mathrm G}_{i}$ are adjacent if they are adjacent in $\Gamma^{*}$.
By \cite[Theorem 1]{PT3}, the subgraphs ${\mathrm G}_{1}$ and ${\mathrm G}_{2}$ are forests.
For (M3), (M4), (M5) and (M7) the above statement fails: 
$z$-monodromies of all faces of the bipyramid $BP_n$ are of type
\begin{enumerate}
\item[$\bullet$] (M3) for $n=2k+1$ where $k$ is odd,
\item[$\bullet$] (M4) for $n=2k+1$ where $k$ is even,
\item[$\bullet$] (M7) for $n=2k$ where $k$ is odd,
\item[$\bullet$] (M5) for $n=2k$ where $k$ is even.
\end{enumerate}

\begin{prop}\label{prop-3}
If $M_F$ is {\rm (M6)}, then $F$ is of type I for any $z$-orientation of $\Gamma$.
\end{prop}
\begin{proof}
Let $e_1,e_2,e_3$ be consecutive oriented edges of the face $F$. We suppose that the $z$-monodromy of $F$ is (M6), i.e.
$$M_{F}=(-e_{1},e_{3},e_{2})(-e_{2},-e_{3},e_{1}).$$
There are precisely two zigzags (up to reversing) which contain $F$
$$e_1,e_2,\dots,-e_1,-e_3,\dots
\;\mbox{and }\; e_2,e_3,\dots\,;$$
since the edge corresponding to the pair $\{e_1,-e_1\}$ is passed in two different directions by the same zigzag, then it is of type I for any orientation of the zigzag.
Therefore, $F$ is of type I for any $z$-orientation.
\end{proof}

\begin{lemma}\label{lemma-2}
Let $F$ be a face in $(\Gamma,\tau)$ such that there are precisely two zigzags from $\tau$ which contain edges from $F$. 
Then the following assertions are fulfilled:
\begin{enumerate}
\item[(1)] There is a unique edge $e\in F$ which occurs in one of these zigzags twice,
\item[(2)] The type of $e$ does not depend on the choice of $z$-orientation,
\item[(3)] If $e$ is of type I, then $M_F$ is {\rm (M6)}. If $e$ is of type II, then $M_F$ is {\rm (M7)}.
\end{enumerate}
\end{lemma}
\begin{proof}
(1). Any face occurs precisely thrice, as a pair of its adjacent edges, in zigzags from the $z$-orientation $\tau$.
By the assumption, there are precisely two zigzags from $\tau$ which pass through our face. 
This is possible only when one of these zigzags passes through it once and the second twice. 

(2). The edge $e$ can occur in the same zigzag twice in two ways: the zigzag passes through $e$ the first time in one of directions and the second time in the opposite (type I) or the zigzag passes through $e$ twice in the same direction (type II). 
It is easy to see that the type of $e$ is the same for any $z$-orientation of $\Gamma$.

(3). By \cite[Remark 2]{PT2} the $z$-monodromy of the face $F$ is (M6) or (M7). 
In the case (M6) the statement follows form Proposition \ref{prop-3}. 
Let $e_1, e_2, e_3$ be consecutive edges of $F$ and $M_F$ be of type (M7), i.e. 
$$M_{F}=(e_{1},e_{2})(-e_{1},-e_{2}).$$
In this case, $F$ occurs twice in the zigzag
$$e_2,e_3,\dots, e_3, e_1, \dots$$
and $e_3$ is of type II for any $z$-orientation of $\Gamma$.
\end{proof}

Now, we can construct a class of toric triangulations, where $z$-monodromies are of type (M6) for all faces. 
Our arguments are based on Lemma \ref{lemma-2}.
\begin{exmp}{\rm
Let $n,m$ be odd numbers not less than $3$ and let $\Gamma_0$ be a $n\times m$ grid where the opposite sides are identified.
Then $\Gamma_0$ can be embedded into a torus in the natural way.
Suppose that $\Gamma=\text{T}(\Gamma_0)$ (see Fig. 11 for the case $n=m=3$).

\begin{center}
\begin{tikzpicture}[scale=0.7, xshift=0cm]

\draw[fill=black] (0,0) circle (1.75pt);
\draw[fill=black] (2,0) circle (1.75pt);
\draw[fill=black] (4,0) circle (1.75pt);
\draw[fill=black] (6,0) circle (1.75pt);

\draw[fill=black] (0,2) circle (1.75pt);
\draw[fill=black] (2,2) circle (1.75pt);
\draw[fill=black] (4,2) circle (1.75pt);
\draw[fill=black] (6,2) circle (1.75pt);

\draw[fill=black] (0,4) circle (1.75pt);
\draw[fill=black] (2,4) circle (1.75pt);
\draw[fill=black] (4,4) circle (1.75pt);
\draw[fill=black] (6,4) circle (1.75pt);

\draw[fill=black] (0,6) circle (1.75pt);
\draw[fill=black] (2,6) circle (1.75pt);
\draw[fill=black] (4,6) circle (1.75pt);
\draw[fill=black] (6,6) circle (1.75pt);

\draw[fill=black] (1,1) circle (1.75pt);
\draw[fill=black] (3,1) circle (1.75pt);
\draw[fill=black] (5,1) circle (1.75pt);

\draw[fill=black] (1,3) circle (1.75pt);
\draw[fill=black] (3,3) circle (1.75pt);
\draw[fill=black] (5,3) circle (1.75pt);

\draw[fill=black] (1,5) circle (1.75pt);
\draw[fill=black] (3,5) circle (1.75pt);
\draw[fill=black] (5,5) circle (1.75pt);

\draw[thick, line width=2pt] (0,0) -- (6,0) -- (6,6) -- (0,6) -- cycle;
\draw[thick, line width=2pt] (2,0) -- (2,6);
\draw[thick, line width=2pt] (4,0) -- (4,6);
\draw[thick, line width=2pt] (0,2) -- (6,2);
\draw[thick, line width=2pt] (0,4) -- (6,4);

\draw[thick] (0,4) -- (4,0);
\draw[thick] (0,2) -- (2,0);
\draw[thick] (0,6) -- (6,0);
\draw[thick] (2,6) -- (6,2);
\draw[thick] (4,6) -- (6,4);

\draw[thick] (0,2) -- (4,6);
\draw[thick] (0,4) -- (2,6);
\draw[thick] (0,0) -- (6,6);
\draw[thick] (4,0) -- (6,2);
\draw[thick] (2,0) -- (6,4);

\draw [->]  (-0.3,0.2) -- (-0.3, 5.8);
\draw [->]  (6.3,0.2) -- (6.3, 5.8);

\draw [->]  (0.2,-0.3) -- (5.8, -0.3);
\draw [->]  (0.2, 6.3) -- (5.8, 6.3);

\node at (1,-0.55) {$e_4$};
\node at (3,-0.55) {$e_5$};
\node at (5,-0.55) {$e_6$};

\node at (1,6.55) {$e_4$};
\node at (3,6.55) {$e_5$};
\node at (5,6.55) {$e_6$};

\node at (-0.57,1) {$e_1$};
\node at (-0.57,3) {$e_2$};
\node at (-0.57,5) {$e_3$};

\node at (6.63,1) {$e_1$};
\node at (6.63,3) {$e_2$};
\node at (6.63,5) {$e_3$};

\end{tikzpicture}
\captionof{figure}{ }
\end{center}
Each zigzag of $\Gamma$ determines a band formed by $n$ or $m$ squares from the grid (see Fig. 12 for a band consisting of $5$ squares) and passes through each face of this band twice.
\begin{center}
\begin{tikzpicture}[scale=0.8, xshift=0cm]

\draw[fill=black] (0,0) circle (1.75pt);
\draw[fill=black] (2,0) circle (1.75pt);
\draw[fill=black] (4,0) circle (1.75pt);
\draw[fill=black] (6,0) circle (1.75pt);
\draw[fill=black] (8,0) circle (1.75pt);
\draw[fill=black] (10,0) circle (1.75pt);

\draw[fill=black] (0,2) circle (1.75pt);
\draw[fill=black] (2,2) circle (1.75pt);
\draw[fill=black] (4,2) circle (1.75pt);
\draw[fill=black] (6,2) circle (1.75pt);
\draw[fill=black] (8,2) circle (1.75pt);
\draw[fill=black] (10,2) circle (1.75pt);

\draw[fill=black] (1,1) circle (1.75pt);
\draw[fill=black] (3,1) circle (1.75pt);
\draw[fill=black] (5,1) circle (1.75pt);
\draw[fill=black] (7,1) circle (1.75pt);
\draw[fill=black] (9,1) circle (1.75pt);

\draw[thick, line width=2pt] (0,0) -- (0,2) -- (10,2) -- (10,0) -- (0,0);
\draw[thick, line width=2pt] (2,0) -- (2,2);
\draw[thick, line width=2pt] (4,0) -- (4,2);
\draw[thick, line width=2pt] (6,0) -- (6,2);
\draw[thick, line width=2pt] (8,0) -- (8,2);

\draw [thick, decoration={markings,
mark=at position 0.48 with {\arrow[scale=2]{<}},
mark=at position 0.63 with {\arrow[scale=2]{>}}},
postaction={decorate}] (0,0) -- (0,2);
\draw [thick, decoration={markings,
mark=at position 0.48 with {\arrow[scale=2]{<}},
mark=at position 0.63 with {\arrow[scale=2]{>}}},
postaction={decorate}] (2,0) -- (2,2);
\draw [thick, decoration={markings,
mark=at position 0.48 with {\arrow[scale=2]{<}},
mark=at position 0.63 with {\arrow[scale=2]{>}}},
postaction={decorate}] (4,0) -- (4,2);
\draw [thick, decoration={markings,
mark=at position 0.48 with {\arrow[scale=2]{<}},
mark=at position 0.63 with {\arrow[scale=2]{>}}},
postaction={decorate}] (6,0) -- (6,2);
\draw [thick, decoration={markings,
mark=at position 0.48 with {\arrow[scale=2]{<}},
mark=at position 0.63 with {\arrow[scale=2]{>}}},
postaction={decorate}] (8,0) -- (8,2);
\draw [thick, decoration={markings,
mark=at position 0.48 with {\arrow[scale=2]{<}},
mark=at position 0.63 with {\arrow[scale=2]{>}}},
postaction={decorate}] (10,0) -- (10,2);

\draw [thick,decoration={markings,
mark=at position 0.6 with {\arrow[scale=2]{>}}},
postaction={decorate}] (0,2) -- (1,1);
\draw [thick,decoration={markings,
mark=at position 0.6 with {\arrow[scale=2]{>}}},
postaction={decorate}] (1,1) -- (2,2);
\draw [thick,decoration={markings,
mark=at position 0.6 with {\arrow[scale=2]{>}}},
postaction={decorate}] (2,2) -- (3,1);
\draw [thick,decoration={markings,
mark=at position 0.6 with {\arrow[scale=2]{>}}},
postaction={decorate}] (3,1) -- (4,2);
\draw [thick,decoration={markings,
mark=at position 0.6 with {\arrow[scale=2]{>}}},
postaction={decorate}] (4,2) -- (5,1);
\draw [thick,decoration={markings,
mark=at position 0.6 with {\arrow[scale=2]{>}}},
postaction={decorate}] (5,1) -- (6,2);
\draw [thick,decoration={markings,
mark=at position 0.6 with {\arrow[scale=2]{>}}},
postaction={decorate}] (6,2) -- (7,1);
\draw [thick,decoration={markings,
mark=at position 0.6 with {\arrow[scale=2]{>}}},
postaction={decorate}] (7,1) -- (8,2);
\draw [thick,decoration={markings,
mark=at position 0.6 with {\arrow[scale=2]{>}}},
postaction={decorate}] (8,2) -- (9,1);
\draw [thick,decoration={markings,
mark=at position 0.6 with {\arrow[scale=2]{>}}},
postaction={decorate}] (9,1) -- (10,2);

\draw [thick,decoration={markings,
mark=at position 0.6 with {\arrow[scale=2]{>}}},
postaction={decorate}] (0,0) -- (1,1);
\draw [thick,decoration={markings,
mark=at position 0.6 with {\arrow[scale=2]{>}}},
postaction={decorate}] (1,1) -- (2,0);
\draw [thick,decoration={markings,
mark=at position 0.6 with {\arrow[scale=2]{>}}},
postaction={decorate}] (2,0) -- (3,1);
\draw [thick,decoration={markings,
mark=at position 0.6 with {\arrow[scale=2]{>}}},
postaction={decorate}] (3,1) -- (4,0);
\draw [thick,decoration={markings,
mark=at position 0.6 with {\arrow[scale=2]{>}}},
postaction={decorate}] (4,0) -- (5,1);
\draw [thick,decoration={markings,
mark=at position 0.6 with {\arrow[scale=2]{>}}},
postaction={decorate}] (5,1) -- (6,0);
\draw [thick,decoration={markings,
mark=at position 0.6 with {\arrow[scale=2]{>}}},
postaction={decorate}] (6,0) -- (7,1);
\draw [thick,decoration={markings,
mark=at position 0.6 with {\arrow[scale=2]{>}}},
postaction={decorate}] (7,1) -- (8,0);
\draw [thick,decoration={markings,
mark=at position 0.6 with {\arrow[scale=2]{>}}},
postaction={decorate}] (8,0) -- (9,1);
\draw [thick,decoration={markings,
mark=at position 0.6 with {\arrow[scale=2]{>}}},
postaction={decorate}] (9,1) -- (10,0);

\draw [->]  (-0.3,0.2) -- (-0.3, 1.8);
\draw [->]  (10.3,0.2) -- (10.3, 1.8);

\end{tikzpicture}
\captionof{figure}{ }
\end{center}
Observe that the edges common for two consecutive squares from the grid are passed twice (they marked on Fig. 12 by the bold line) and are of type I for any $z$-orientation. 
Remaining edges are passed by the zigzag once. 
Therefore, all edges of subgraph $\Gamma_0$ are of type I and all faces of $\Gamma$ are of type I for any $z$-orientation.
It is clear that any edge incident to a vertex in the interior of a square occurs once in two different zigzags. 
Thus, for any face of $\Gamma$ there are precisely two zigzags which pass it. 
Lemma \ref{lemma-2} guarantees that $z$-monodromies of all faces of $\Gamma$ are (M6).
}\end{exmp}

\subsection*{Acknowledgment}
The author is grateful to Mark Pankov for useful discussions.


\begin{thebibliography}{99}

\bibitem{BCMMcK}
Bonnington C.P. , Conder M., Morton M. and McKenna P., 
{\it Embedding digraphs on orientable surfaces}, J. Combin. Theory Ser. B 85(2002), 1--20.

\bibitem{BHS}
Bonnington C.P., Hartsfield N. and {\v S}ir\'a{\v n} J. 
{\it Obstructions to directed embeddings of Eulerian digraphs in the plane}, European J. Combin. 25(2004), 877--891.

\bibitem{BD}
Brinkmann G., Dress, A. W. M.,
{\it PentHex puzzles. A reliable and efficient top-down approach to fullerene-structure enumeration},
Adv. Appl. Math. 21(1998), 473--480.

\bibitem{CGH}
Chen Y., Gross J. L., and Hu X., 
{\it Enumeration of digraph embeddings}, European J. Combin.  36(2014), 660--678.

\bibitem{CrRos}
Crapo H., Rosenstiehl P., {\it On lacets and their manifolds}, Discrete Math. 233 (2001), 299--320.

\bibitem{Coxeter} Coxeter H.S.M., 
{\it Regular polytopes}, 
Dover Publications, New York 1973 (3rd ed).

\bibitem{DDS-book} 
Deza M., Dutour Sikiri\'c M., Shtogrin M.,
{\it Geometric Structure of Chemistry-relevant Graphs: zigzags and central circuit}, 
Springer 2015.

\bibitem{Farr}
Farr G. E., {\it Minors for alternating dimaps}, 
Q. J. Math. 69(2018), 285--320 .

\bibitem{GR-book}
Godsil C., Royle G., {\it Algebraic Graph Theory}, 
Graduate Texts in Mathematics 207, Springer 2001.

\bibitem{Lins1} 
Lins S., {\it Graph-encoded maps}, J. Combin. Theory, Ser. B 32(1982), 171--181.

\bibitem{Lins2} 
Lins S., Oliveira-Lima E., Silva V., {\it A homological solution for the Gauss code problem in arbitrary surfaces}, 
J. Combin. Theory, Ser. B 98(2008), 506--515.

\bibitem{McCourt}
McCourt T. A.,
{\it Growth rates of groups associated with face 2-coloured triangulations and directed Eulerian digraphs on the sphere},
Electron. J. Comb. 23(2016), P1.51, 23 p. 

\bibitem{McMSch}
McMullen P., Schulte E., {\it Abstract Regular Polytopes}, Cambridge University Press 2002.

\bibitem{MT-book}
Mohar B., Thomassen C., {\it Graphs on Surfaces}, The Johns Hopkins University Press 2001.

\bibitem{PT1} 
Pankov M., Tyc A., {\it Connected sums of z-knotted triangulations}, 
Euro. J. Comb. 80(2019), 326-338

\bibitem{PT2}
Pankov M., Tyc A., {\it $Z$-knotted triangulations of surfaces},
accepted to Discrete Comput. Geom.

\bibitem{PT3} 
Pankov M., Tyc A., {\it On two types of $z$-monodromy in triangulations of surfaces},
Discrete Math. 342(2019), 2549-2558.

\bibitem{Shank} Shank H., 
{\it The theory of left-right paths} in Combinatorial Mathematics III,
Lecture Notes in Mathematics 452, Springer 1975, 42--54.

\end{thebibliography}
\end{document}